\documentclass{article}
\usepackage{graphicx} 
\usepackage{amsmath,array,amsfonts,titlesec,amsthm,graphics,setspace,caption,epsf,cite}
\usepackage{algorithmic,bibcheck,pdfpages}
\usepackage[ruled,vlined]{algorithm2e}
\usepackage{url, epstopdf, amscd, amssymb, color,fancyhdr,nomencl,colortbl}
\usepackage[colorlinks]{hyperref}
\usepackage{longtable,latexsym,epsfig,caption,subcaption,microtype,textcomp,stmaryrd,listings}
\hypersetup{linkcolor=red,citecolor=blue}
\usepackage{nomencl, array, amsfonts, setspace, epsf}

\usepackage{float}

\usepackage{amssymb}

\usepackage{cancel}

\usepackage{graphicx,wrapfig,lipsum}

\usepackage{multicol}
\numberwithin{equation}{section}
\newtheorem{thm}{Theorem}[section]
 
 \newtheorem{lem}[thm]{Lemma}
 \newtheorem{cor}[thm]{Corollary}
 \newtheorem{exm}{Example}[section]
 \newtheorem{rem}{Remark}[section]
 \newtheorem{dfn}{Definition}[section]

\usepackage[
  a4paper,
  left=2.5cm,
  right=2.5cm,
  top=2.5cm,
  bottom=2.5cm
]{geometry}

\newcounter{algorithmbis}
\renewcommand{\thealgorithmbis}{\thesection.\arabic{algorithmbis}}
\def\algorithmbis{\@ifnextchar[{\@algorithmbisa}{\@algorithmbisb}}
\def\@algorithmbisa[#1]{%
  \refstepcounter{algorithmbis}
  \trivlist
  \leftmargin\z@
  \itemindent\z@
  \labelsep\z@
  \item[\parbox{\textwidth}{%
    \hrule
    \hrule
    \noindent\strut\textbf{Algorithm \thealgorithmbis} #1
    \hrule
  }]\hfil\vskip0em%
}
\def\@algorithmbisb{\@algorithmbisa[]}

\makeatother

\def\enddemo{\qed \endtrivlist}
\expandafter\let\csname enddemo*\endcsname=\enddemo

\def\qedsymbol{\ifmmode\bgroup\else$\bgroup\aftergroup$\fi
  \vcenter{\hrule\hbox{\vrule
height.6em\kern.6em\vrule}\hrule}\egroup}
\def\qed{\ifmmode\else\unskip\nobreak\fi\quad\qedsymbol}

\SetKwProg{Fn}{Function}{}{}
\newcounter{slika}[section]

\newcommand{\diag}{\mathrm{diag}}

\newcommand{\N}{\mathbb{N}}

\renewcommand{\O}{\mathcal{O}}

\newcommand{\tr}{\textup{tr}}
\newcommand{\R}{\mathbb{R}}

\newcommand{\seqn}[1]{\left({#1}\right)_{k=0}^\infty}

\newcommand{\bb}{\begin{equation}}
\newcommand{\ee}{\end{equation}}

\numberwithin{equation}{section}

\begin{document}

\title{\bf Algorithm with variable coefficients for computing matrix inverses}


\author{Marko V. Kostadinov$^1$, Mihailo Krsti\' c$^2$, Kostadin Rajkovi\' c$^3$, Marko D. Petkovi\' c$^4$,\\ \\
$^1$\footnotesize{University of Ni\v s, Faculty of Mechanical Engineering, Aleksandra Medvedeva 14, Ni\v{s}, 18 000, Serbia} \\
\footnotesize{\small \texttt{E-mail: marko.kostadinov@masfak.ni.ac.rs}}
\\
$^2$\footnotesize{University of Belgrade, Faculty of Mathematics, Studentski Trg 16, Belgrade, 11 000, Serbia}\\
\footnotesize{\small \texttt{E-mail: mihailo.krstic@matf.bg.ac.rs}}
\\
$^3$\footnotesize{University of Ni\v s, Faculty of Sciences and Mathematics, Vi\v segradska 33, Ni\v{s}, 18 000,Serbia}\\
\footnotesize{\small \texttt{E-mail: kosta.rajkovic@gmail.com}}
\\
$^4$\footnotesize{University of Ni\v s, Faculty of Sciences and Mathematics, Vi\v segradska 33, Ni\v{s}, 18 000,Serbia}\\
\footnotesize{\small \texttt{E-mail: dexterofnis@gmail.com}} 
\\
}

\date{\today}
\renewcommand{\theequation}{\thesection.\arabic{equation}}
\maketitle
\begin{abstract}
We present a new approach for the construction of the efficient generalized Schultz
iterative methods with variable coefficients, for computing the inverse matrix. The proposed method has the form 
$$
X_{k+1} = X_k((\alpha_k+\beta_k) I+\beta_k F_k),\quad F_k=I-AX_k,
$$
where $A \in \R^{n\times n}$ is real matrix, $X_0\in \R^{n\times n}$ is initial approximation and $\alpha_k, \beta_k \in \R$ are the coefficients determined in each iteration (i.\,e. for every $k\in\{0,1,2,\ldots\}$). They are chosen so that the residual matrix of the next iteration $F_{k+1}=I-AX_{k+1}$ has the minimal Frobenius norm. Convergence and robustness analysis are performed theoretically and the necessary and/or sufficient conditions for the convergence are derived. The new method is compared with existing ones in the series of numerical experiments. The results obtained support the theoretical conclusions on convergence and show the advantage of the new method in terms of computational efficiency. Moreover, it shows that the new method converges for much wider class of initial approximations $X_0\in \R^{n\times n}$. \\

\medskip

\medskip


\medskip

\noindent \textbf{Corresponding author}: Marko D. Petkovi\'c

\end{abstract}
\medskip

\section{Introduction and preliminaries}

\qquad The problem of computing the inverse matrix $A^{-1}$ for a given invertible real matrix $A$ has played a central role in the development of numerical linear algebra. 
Since the early days of matrix theory, researchers have sought efficient and numerically stable procedures for determining matrix inverses, particularly for large-scale problems arising in scientific computing, engineering, and applied mathematics.

\medskip

The vector space of all real matrices of dimension $n\times m$ is denoted by $\R^{n\times m}$. For a matrix $Q\in\R^{n\times m}$, we denote by $\mathcal{R}(Q)$ its image (range) and by $\mathcal{N}(Q)$ its kernel (null space). In the special case when $m=n$, the symbols $\tr(Q)$ and $\sigma(Q)$ represent the trace and the spectrum of the matrix $Q$, respectively. The spectral radius of $Q$ is defined by
$
r(Q) = \max\{ |\lambda|\, :\, \lambda \in \sigma(Q)\}.$
A commonly used norm on the vector space $\R^{n\times m}$ is the \textit{Frobenius norm} $\|\cdot\|_F$, which is given by
\begin{equation}\label{Fro}
\|Q\|_F=\sqrt{\sum_{i=1}^n\sum_{j=1}^m q_{ij}^2},\qquad Q=[q_{ij}]\in\R^{n\times m}.
\end{equation}
This norm is particularly convenient in numerical linear algebra applications because it can be computed directly from \eqref{Fro} with computational complexity $\O(nm)$. Such efficiency is not typical for many other matrix norms, including the spectral norm $\|\cdot\|_2$, which is defined only for square matrices.

\medskip

A particularly influential idea for the matrix inversion, was the use of the Neumann series expansion. This approach inspired a variety of iterative methods, including the well-known Schultz (or Newton) iteration introduced in the mid-twentieth century. 
Later developments generalized this idea into the broader class of \textit{Hyper-Power methods}, which achieve higher orders of convergence by incorporating higher-degree polynomial corrections. The pioneering work on this topic was \cite{Sch}.

\medskip

These approaches rely on truncating the infinite series expansion
$
(I-M)^{-1}=I+M+M^2+M^3+\cdots$
which is valid for any invertible matrix $M\in \R^{n\times n}$ whose spectral radius satisfies $r(M)<1$. From this expansion, one derives the well known Hyper-Power iterative scheme
\begin{equation}\label{HP}
X_{k+1}=X_k(I+F_k+\cdots+F_k^{d}),\qquad k\in\N_0
\end{equation}
for any integer $d>1$, where $F_k=I-AX_k$ denotes the residual matrix. The method converges provided that the initial matrix $X_0\in \R^{n\times n}$ satisfies $r(F_0)<1$, and its order of convergence equals $d+1$. The usual choice for $X_0$ is $X_0=\frac{2}{\|A\|_F^2} \cdot A^T$ which clearly satisfies $r(F_0)<1$.
A special case of \eqref{HP} is obtained by choosing $d=1$. This yields the classical \textit{Schultz (Newton) iteration} (see \cite{Sch})
\begin{equation} \label{Sch}
    X_{k+1}=X_k(I+F_k)=X_k(2I-AX_k),\qquad k\in\N_0, 
\end{equation}
which will be referred to as \textbf{HP2}. Furthermore, the Hyper-Power method corresponding to $d=2$ is 
\begin{equation} \label{eq:HP3}
    X_{k+1}=X_k(1+F_k+F_k^2),\qquad k\in\N_0, 
\end{equation}
and denoted by \textbf{HP3}.

\medskip

Further research expanded these techniques to compute not only classical matrix inverses but also generalized inverses such as the Moore--Penrose and Drazin inverses, which are essential in the analysis of singular or rectangular matrices. 
It has been demonstrated that the Schultz iteration \eqref{Sch} can also be used for computing various generalized inverses of matrices that may be singular or rectangular (see, for instance, \cite{SMSMP, SMSDrazin}). Related results for the general Hyper-Power method \eqref{HP} are presented in \cite{ChenWan,Weiguo,Liu}.

\medskip

A broader family of methods, known as \textit{Generalized Schultz iterative methods}, have the form $X_{k+1}=X_kP(AX_k)$, where $P\in \R[x]$ is a real polynomial of degree $d$. This family is introduced and analyzed by Petkovi\'c in \cite{MP1}, and contains all previously mentioned iterative methods, for the special choices of the polynomial $P$. Iterations can be written in the equivalent form 
$
X_{k+1} = X_k \bar P(F_k)$,
where $\bar P(x)=P(1-x)$ is also $d$-degree real polynomial and $F_k=I-AX_k$ is the matrix residual. Theoretical properties, including the convergence order, sufficient conditions for convergence based on the initial matrix $X_0$ as well the behaviour when matrix $A$ is singular, are extensively studied in \cite{MP1}. 

\medskip

Numerous variants of the Hyper-Power method and other Generalized Schultz iterative methods, have been proposed by selecting different polynomial coefficients and computation procedure, in order to improve the convergence properties and computational efficiency of the resulting iterative schemes (see for example \cite{PKR}).
The Hyper-Power method \eqref{HP} is obtained by choosing $\bar P(x)=1+x+\cdots+x^d$. Its order of convergence is equal to $d+1$, while each iteration requires $d+1$ matrix--matrix multiplications (one multiplication to compute $AX_k$, $d-1$ multiplications to evaluate the powers $F_k^2,\ldots,F_k^d$ and one final multiplication by $X_k$). Consequently, increasing $d$ simultaneously increases both the convergence order and the computational cost of each iteration. In practice, the choice $d=2$ is generally regarded as the most favorable compromise between these two competing requirements, and the corresponding \textbf{HP3} method has become one of the most widely used members of the Hyper-Power family. 
A representative example of the highly efficient modification of the Hyper-Power method for $d=8$, is \textbf{IHP9} method from \cite{MP2} which has convergence order $9$ with only 5 matrix multiplications per iteration.

\medskip

\medskip
The research on iterative methods for computing inverse matrix and generalized inverses is still very active research topic. Some recent papers are \cite{Aceto2025, BB, Art, Cordero2021, Cordero2024, Seven, XiaLi,ZhangKangTan, Kaur2020, effgen, Martinez, GIBS}. In \cite{BB}, the authors studied a variation of the Schultz iterative method with an additional focus on ill-posed problems and obtaining stable solutions to them. 
In \cite{Cordero2021} parametric families iterative methods for the computation of generalized inverses of arbitrary order of convergence are studied with an additional analysis of the choice of initial approximation. 
A further research on this topic is done in \cite{Cordero2024} with the novel inclusion of a dynamical analysis of the methods. An iterative method for finding the generalized Drazin inverse of a matrix of order seven of order seven was studied in \cite{Seven}. In \cite{Kaur2020} a novel hyper-power iterative method for the calculation of the weighted Moore-Penrose inverse was given with the main objective of minimizing the computational complexity. A similar objective can be seen in \cite{effgen}, where an iterative method of order nine was constructed to compute the generalized outer inverse of a matrix using only $7$ matrix multiplications. An accuracy-focused method for computing the Moore-Penrose inverse was studied in \cite{Martinez}. In \cite{GIBS} the Schultz iterative method has been generalized to sparse matrices. In \cite{XiaLi} the authors introduced a modified BAS (block-diagonal and anti-block-diagonal splitting) method for solving the absolute value matrix equation. In \cite{ZhangKangTan}, a two-parameter iterative method for solving the algebraic Riccati equation in the non-symmetric case was studied. In \cite{Aceto2025} an iterative method for computing the inverse of $\varphi$-functions $\psi_l(A)=\left(\varphi_l(A)\right)^{-1}$ was introduced in the special case when $A$ is a large sparse matrix whose eigenvalues lie in the open left half-plane. It is important to mention that novel iterative methods for the computation of inverse matrix and generalized inverses have found practical applications such as MIMO detection systems, see \cite{MIMO,Yang2023}, data influence estimation, see \cite{Zhou2025}, high performance computing \cite{Janna2025}, etc.

\medskip
A fundamentally different approach was introduced for the first time in \cite{OPM}, where the coefficients of the polynomial $\bar P(x)$ are no longer fixed throughout the iterative process, but are allowed to vary from one iteration to the next. More precisely, at each iteration a new polynomial is constructed by selecting its coefficients according to the current iterate. However, the method proposed in \cite{OPM} is restricted to first-degree polynomials, so that only a single coefficient is determined adaptively at each iteration.
More concretly, in \cite{OPM}, authors analysed the following variant of the Schultz method, denoted by \textbf{OPM} method:
\[
X_{k+1}=X_k\bigl(I+(\alpha_k-1)F_k\bigr),\qquad F_k=I-AX_k,\qquad k\in\N_0,
\]
with variable coefficient $ \alpha_k$. 
It is computed by the elements of $F_k$ and $G_k=F_k^2$ using total $\O(n^2)$ operations ($n$ is the dimension of matrix $A$). 
The analysis in \cite{OPM} show that the \textbf{OPM} method outperforms the classical Schultz iterative method, in terms of both total computation time and total number of matrix multiplications.

\medskip


\medskip
In this paper, we extend the idea from \cite{OPM} and derive an iterative method for computing matrix inverse $A^{-1}$ with two variable coefficients. The paper is organized as follows. In Section~\ref{construction}, a new iterative method for matrix inversion is introduced and formulated in the form of Algorithm~\ref{optcoefit} and Algorithm~\ref{algor0}. 
The convergence analysis and error bounds are established in Section~\ref{convergence}. Numerical experiments are given in Section \ref{sect:testing} and demonstrate that the proposed method outperforms existing schemes in terms of both computational speed and accuracy. Finally, concluding remarks are given in Section~\ref{sect:conclusion}.

\section{Construction of the algorithm}
\label{construction}

Let $A\in \R^{n\times n}$ is invertible matrix. Consider the following iterative method
\begin{equation}
X_{k+1}=X_k\left(a_0^{(k)}I+a_1^{(k)}AX_k\right),\qquad k\in\N_0,
\label{eq:Xkak}
\end{equation}
where $X_0 \in \R^{n\times n}$ is given initial matrix, and the coefficients $a_0^{(k)}$ and $a_1^{(k)}$ are being computed in each iteration.

\medskip

Denote by
$
 F_k=I-AX_k
$
the residual matrix in $k$-th iteration of the iterative process. 
The goal is to compute the coefficients $a_0^{(k)}$ and $a_1^{(k)}$ such that the Frobenius norm of $F_{k+1}$ is minimized. 
By direct computation, one obtains:
\begin{align*}\label{Smena2}\|F_{k+1}\|_F^2 &=\|I-AX_{k+1}\|_F^2=\left\|I-AX_k\left(a_0^{(k)}I+a_1^{(k)}F_k\right)\right\|_F^2 \\ &=
\left\|I-(I-F_k)\left(a_0^{(k)}I+a_1^{(k)}F_k\right)\right\|_F^2 =
\left\|\left(1-a_0^{(k)}\right)I+\left(a_0^{(k)}-a_1^{(k)}\right)F_k+a_1^{(k)}F_k^2\right\|_F^2.
\end{align*}
Now introduce the new variables $\alpha_k$ and $\beta_k$ as 
$$
\alpha_k:=a_0^{(k)}-a_1^{(k)},\qquad \beta_k:=a_1^{(k)}
$$
for $k\in \N_0$.
The residual matrix $F_{k+1}$ is now given by
\begin{equation}\label{smena2}
  F_{k+1}=(1-(\alpha_k+\beta_k))I+\alpha_kF_k+\beta_kF_k^2  
\end{equation}
while the main recurrence \eqref{eq:Xkak} 
has the form
\bb
X_{k+1} = X_k ((\alpha_k + \beta_k) I + \beta_k F_k).
\label{eq:Xkmain}
\ee
Denote by $F_k=\left[ f_{i,j}^{(k)}\right]\in \R^{n\times n}$ and $G_k=F^2_k=\left[g_{i,j}^{(k)} \right] \in \R^{n \times n},$
so the equation \eqref{smena2} implies 
\begin{equation}\label{glavna}||F_{k+1}||^2_F= \sum_{\substack{i,j=1\\i\neq j}}^n\left[ \alpha_kf^{(k)}_{i,j}+\beta_kg^{(k)}_{i,j} \right]^2+\sum_{i=1}^n\left[(1+\left(f^{(k)}_{i,i}-1\right)\alpha_k+\left(g^{(k)}_{i,i}-1\right)\beta_k\right]^2.
\end{equation}
A minimization of previous expression by variables $\alpha_k$ and $\beta_k$ is now required to be done. 
For this purpose, consider the auxiliary function $G:\mathbb{R}^2\to\mathbb{R}$ defined by
$$G(x,y)=\sum_{\substack{i,j=1\\i\neq j}}^n\left[ x\cdot f^{(k)}_{i,j}+y\cdot g^{(k)}_{i,j} \right]^2+\sum_{i=1}^n\left[1+\left(f^{(k)}_{i,i}-1\right)\cdot x+\left(g^{(k)}_{i,i}-1\right)\cdot y\right]^2,\qquad (x,y)\in\mathbb{R}^2.
$$
According to \eqref{glavna}, we have $\|F_{k+1}\|^2_F = G(\alpha_k,\beta_k)$. Hence, the
coefficients $\alpha_k$ and $\beta_k$ have to be computed as the minimum of the function $G$. 
Since $G$ is convex quadratic function, they are computed as the solution of the following system of linear equations
$$
\frac{\partial G\left(\alpha_k,\beta_k\right)}{\partial x}=0,\qquad
\frac{\partial G\left(\alpha_k,\beta_k\right)}{\partial y}=0.
$$
By grouping coefficients with variables $\alpha_k$ and $\beta_k$ from previous two equations, we have the next system of linear equations:
\begin{equation*}
\begin{split}\label{sis1}\left[\sum_{\substack{i,j=1\\i\neq j}}^n\left(f_{i,j}^{(k)}\right)^2  +\sum_{i=1}^n\left(f_{i,i}^{(k)}-1\right)^2\right]\alpha_k + \left[\sum_{\substack{i,j=1\\i\neq j}}^n f_{i,j}^{(k)} g_{i,j}^{(k)}  +\sum_{i=1}^n\left(f_{i,i}^{(k)}-1\right)\left(g_{i,i}^{(k)}-1\right)\right]\beta_k=n-\sum_{i=1}^n f_{i,i}^{(k)},\end{split}
\end{equation*}

\begin{equation*}
\begin{split}\label{sis2}\left[\sum_{\substack{i,j=1\\i\neq j}}^n f_{i,j}^{(k)}g_{i,j}^{(k)} +\sum_{i=1}^n\left(f_{i,i}^{(k)}-1\right)\left(g_{i,i}^{(k)}-1\right)\right]\alpha_k + \left[\sum_{\substack{i,j=1\\i\neq j}}^n(g_{i,j}^{(k)} )^2  +\sum_{i=1}^n\left(g_{i,i}^{(k)}-1\right)^2\right]\beta_k=n-\sum_{i=1}^n g_{i,i}^{(k)}.\end{split}
\end{equation*}
The same system can be written in the following equivalent matrix form
\bb
\left[
\begin{array}{cc}
\displaystyle
\sum_{\substack{i,j=1\\i\neq j}}^n\left(f_{i,j}^{(k)}\right)^2+
\sum_{i=1}^n\left(f_{i,i}^{(k)}-1\right)^2
&
\displaystyle
\sum_{\substack{i,j=1\\i\neq j}}^n f_{i,j}^{(k)}g_{i,j}^{(k)}+
\sum_{i=1}^n\left(f_{i,i}^{(k)}-1\right)\left(g_{i,i}^{(k)}-1\right)
\\[1.2em]
\displaystyle
\sum_{\substack{i,j=1\\i\neq j}}^n f_{i,j}^{(k)}g_{i,j}^{(k)}+
\sum_{i=1}^n\left(f_{i,i}^{(k)}-1\right)\left(g_{i,i}^{(k)}-1\right)
&
\displaystyle
\sum_{\substack{i,j=1\\i\neq j}}^n\left(g_{i,j}^{(k)}\right)^2+
\sum_{i=1}^n\left(g_{i,i}^{(k)}-1\right)^2
\end{array}
\right]
\left[
\begin{array}{c}
\alpha_k\\
\beta_k
\end{array}
\right]
=
\left[
\begin{array}{c}
\displaystyle n-\sum_{i=1}^n f_{i,i}^{(k)}\\[0.8em]
\displaystyle n-\sum_{i=1}^n g_{i,i}^{(k)}
\end{array}
\right],
\label{eq:systalphabeta}
\ee
By solving above linear system of equations, one can get the optimal coefficients $\alpha_k$ and $\beta_k$ in each iteration, in the symbolic form. For such purpose denote
\begin{align}
B_0^{(k)} &= n-\sum_{i=1}^n f_{i,i}^{(k)},\label{koefreal1}\\
C_{00}^{(k)} &= \sum_{\substack{i,j=1\\i\neq j}}^n\left(f_{i,j}^{(k)}\right)^2  +\sum_{i=1}^n\left(f_{i,i}^{(k)}-1\right)^2,\\
C_{01}^{(k)} &= C_{10}^{(k)} =\sum_{\substack{i,j=1\\i\neq j}}^n f_{i,j}^{(k)} g_{i,j}^{(k)} +\sum_{i=1}^n\left(f_{i,i}^{(k)}-1\right)\left(g_{i,i}^{(k)}-1\right),\\
B_1^{(k)} &=n-\sum_{i=1}^n g_{i,i}^{(k)},\\
C_{11}^{(k)} &=\sum_{\substack{i,j=1\\i\neq j}}^n\left(g_{i,j}^{(k)}\right)^2  +\sum_{i=1}^n\left(g_{i,i}^{(k)}-1\right)^2.\label{koefreal2}
\end{align} 
for every $k\in \N_0$. 
The desired solution formulas can now be wrtitten as
\begin{align}\label{koef}
\alpha_k &=\frac{C_{11}^{(k)}B_0^{(k)}-C_{01}^{(k)}B_1^{(k)}}{C_{00}^{(k)}C_{11}^{(k)}-\left(C_{01}^{(k)}\right)^2},\qquad
\beta_k=\frac{C_{00}^{(k)}B_1^{(k)}-C_{10}^{(k)}B_0^{(k)}}{C_{00}^{(k)}C_{11}^{(k)}-\left(C_{01}^{(k)}\right)^2}.
\end{align} 
The previous expressions are valid in the case when the system \eqref{eq:systalphabeta} has unique solution, i.\,e. when $D_k=C_{00}^{(k)}C_{11}^{(k)}-\left(C_{01}^{(k)}\right)^2 \neq 0$. 

\medskip

If that is not true, one has to choose the coefficients $\alpha_k$ and $\beta_k$ in a different way. One possible idea is to use the solution of \eqref{eq:systalphabeta} obtained by Moore-Penrose inverse of the system matrix. 

\medskip

However, the empirical study showed that the best approach (in terms of total number of iterations) is use the values $\alpha_k=0$ and $\beta_k=1$, corresponding to the well-known Schultz method \eqref{Sch}. It is done whenever the absolute value of the denominator $D_k$ is less than $\delta |C_{00}^{(k)}C_{11}^{(k)}|$, where $\delta>0$ is the tolerance parameter. The complete pseudocode for computing the coefficients $\alpha_k$ and $\beta_k$ is given by the Algorithm \ref{optcoefit}.


\medskip

\begin{algorithm}[H]
\DontPrintSemicolon
\KwData{Matrices $F_k, G_k\in \R^{n\times n}$ and a tolerance $\delta>0$.}
\KwResult{The coefficients $\alpha_k$ and $\beta_k$.}

Compute
$C_{00}^{(k)}$, $C_{01}^{(k)}$, $C_{10}^{(k)}$, $C_{11}^{(k)}$,
$B_0^{(k)}$ and $B_1^{(k)}$
using expressions \eqref{koefreal1}--\eqref{koefreal2}.

$D_k := C_{00}^{(k)}C_{11}^{(k)} - C_{01}^{(k)}C_{10}^{(k)}$.

\If{$|D_k| \geqslant \delta |C_{00}^{(k)}C_{11}^{(k)}|$}{
    \[
        \alpha_k :=
        \frac{C_{11}^{(k)}B_0^{(k)} - C_{01}^{(k)}B_1^{(k)}}{D_k},
        \qquad
        \beta_k :=
        \frac{C_{00}^{(k)}B_1^{(k)} - C_{10}^{(k)}B_0^{(k)}}{D_k}.
    \]
}
\Else{
    \[
        \alpha_k := 0,
        \qquad
        \beta_k = 1.
    \]
}

\Return{$\alpha_k$, $\beta_k$}

\caption{Algorithm for the coefficients $\alpha_k$ and $\beta_k$ in each iteration}
\label{optcoefit}
\end{algorithm}

\medskip

Now we are ready to formulate the new iterative method based on the recurrence \eqref{eq:Xkmain}, which uses the coefficients $\alpha_k$ and $\beta_k$ determined by the Algorithm \ref{optcoefit}. The pseudocode of the new method \textbf{SSHP2} (Stable Scaled Hyper-power method of degree 2) is given by the Algorithm \ref{algor0}. 


\begin{algorithm}[H]\label{algor0}
\DontPrintSemicolon
 \KwData{Invertible matrix $A\in \R^{n\times n}$, initial approximation $X_0\in \R^{n\times n}$, tolerance $\delta>0$ and the precision $\epsilon>0$.}
 \KwResult{The approximate inverse matrix $A^{-1}$.}
 $F_0:=I-AX_0$\;
 $k:=0$\;
 \While{$||F_k||_F\geqslant \epsilon$}{
    $G_k:=F_k^2$\;
	Compute $\alpha_k$ and $\beta_k$ using Algorithm \ref{optcoefit} (for $F_k$, $G_k$ and $\delta$).
    
	$X_{k+1}:=X_k \left(\left(\alpha_k+\beta_k\right)I+\beta_kF_k\right)$\; 
    $k:=k+1$
    
    $F_k:=I-AX_k$\;
}
\Return $X_{k}$
 \caption{\textbf{SSHP2} method}
\end{algorithm}

\medskip 

The total time complexity of Algorithm \ref{optcoefit} is $\O(n^2)$, according to the expressions \eqref{koefreal1}-\eqref{koefreal2}. This is asymptotically less complex than matrix-matrix multiplication, which still remains predominant operation in the \texttt{while} loop of Algorithm \ref{algor0}. 

\medskip



We can summarize the results of this section into the following theorem.


\begin{thm}
Consider the invertible matrix $A\in \R^{n\times n}$ and the matrix $X_k \in \R^{n\times n}$ ($k\in \N_0$). Let $F_k=I-AX_k$, $G_k=F_k^2$ and $F_{k+1}=I-AX_{k+1}$, where $X_{k+1}$ is determined by \eqref{eq:Xkmain}. If the system \eqref{eq:systalphabeta} is non-singular, then
$$
\min_{\alpha_k, \beta_k\in \mathbb{R}} \|F_{k+1}\|_F^2,
$$
is achieved for the coefficients $\alpha_k$ and $\beta_k$ chosen from \eqref{koefreal1}-\eqref{koefreal2}.
\end{thm}

\section{Convergence analysis}\label{convergence}

Throughout this section, we assume that we work with real matrices and that the initial approximation is of the form $X_0=\alpha A^T$, where $0<\alpha<2/\|A\|_2^2$ is a suitable constant, and $A\in \R^{n\times n}$ is the invertible matrix whose inverse we want to find. In accordance with this, the fact that $A$ is invertible will not be explicitly stated in the results of this section.

\medskip

The main point is that the convergence analysis of our method can be treated in terms of the eigenvalues of the matrix $F_k$. Notice that for such a choice of $X_0$ the sequence $(F_k)_{k=0}^\infty$ consists of real and symmetric matrices, implying that the eigenvalues of $F_k$ are real. In our analysis we will consider the equivalent and simpler form of the method. 
Let us denote $$\sigma(F_k)=\{\lambda_{i,k}:i = 1,2,\ldots,n\},\qquad k\in\N_0.$$ 
From \eqref{smena2}  we get (see \cite[Theorem 1.1.6]{Horn}):
\begin{equation}\label{vezab0b1}\lambda_{i,k+1}= 1-(\alpha_k+\beta_k) + \alpha_k\lambda_{i,k} +\beta_k\lambda_{i,k}^2,\qquad i\in\{1,\ldots,n\}.\end{equation}
Since $(F_k)_{k=0}^\infty$ is a sequence of real symmetric matrices we can rewrite the formulas \eqref{koef} in more suitable form. Starting from the equalities \eqref{koefreal1}-\eqref{koefreal2} we get:
$$
\begin{aligned}
   B_0^{(k)} &= n-\sum_{i=1}^n f_{i,i}^{(k)} = \sum_{i=1}^n \left(1-f_{i,i}^{(k)}\right) = \tr (I-F_k),\\[3pt]
   B_1^{(k)} & = n-\sum_{i=1}^n g_{i,i}^{(k)} = \sum_{i=1}^n \left(1-g_{i,i}^{(k)}\right) = \tr(I-F_k^2),\\[3pt]
   C_{00}^{(k)} &= \sum_{\substack{i,j=1 \\i\neq j}}^n\left(f_{i,j}^{(k)}\right)^2  +\sum_{i=1}^n\left(f_{i,i}^{(k)}-1\right)^2 = \sum_{i=1}^n \sum_{j=1}^n \left(f_{i,j}^{(k)}\right)^2 + \sum_{i=1}^n \left(1-2f_{i,i}^{(k)} \right) \\
   & = \tr( F_k^T F_k) + \tr (I-2F_k) = \tr(I-2F_k + F_k^2) = \tr((I-F_k)^2) = \| I-F_k\|_F^2,\\[3pt]
   C_{11}^{(k)} &  =\sum_{\substack{i,j=1\\i\neq j}}^n\left(g_{i,j}^{(k)}\right)^2  +\sum_{i=1}^n\left(g_{i,i}^{(k)}-1\right)^2 =  \sum_{i=1}^n \sum_{j=1}^n \left(g_{i,j}^{(k)}\right)^2 + \sum_{i=1}^n \left(1-2g_{i,i}^{(k)}\right) \\
   & = \tr
   \left( (F_k^2)^T F_k^2\right) + \tr (I-2F_k^2) = \tr(I-2F_k^2 + F_k^4) = \tr((I-F_k^2)^2) = \| I-F_k^2\|_F^2,\\[3pt]
   C_{01}^{(k)} & = C_{10}^{(k)}  =\sum_{\substack{i,j=1\\i\neq j}}^n f_{i,j}^{(k)} g_{i,j}^{(k)} +\sum_{i=1}^n\left(f_{i,i}^{(k)}-1\right)\left(g_{i,i}^{(k)}-1\right)\\&=\sum_{i=1}^n \sum_{j=1}^n f_{i,j}^{(k)}g_{i,j}^{(k)} + \sum_{i=1}^n \left( 1- f_{i,i}^{(k)} - g_{i,i}^{(k)} \right)  \\
   & = \tr ( F_k^T F_k^2) + \tr (I-F_k-F_k^2)= \tr (I-F_k-F_k^2+F_k^3)\\& = \tr( I-F_k - F_k^2 (I-F_k)) = \tr ((I-F_k)(I-F_k^2)),
\end{aligned}
$$
Now substituting in \eqref{koef}, gives us
\bb \label{alfabeta1}
\aligned
\alpha_k & = \frac{- \| F_k^2 - I \|_F^2 \tr (F_k-I) + \tr(F_k^2 -I)\tr ( (F_k^2-I)(F_k-I)) }{ \| F_k -I \|_F^2 \| F_k^2 - I \|_F^2 - \left( \tr ( (F_k^2-I)(F_k-I)) \right)^2} \\
& = \frac{ \| I-F_k^2\|_F^2 \tr (I-F_k) - \tr(I-F_k^2)\tr ( (I-F_k^2)(I-F_k)) }{ \| I- F_k\|_F^2 \| I- F_k^2\|_F^2 - \left( \tr ( (I-F_k^2)(I-F_k)) \right)^2}, \\[6pt]
\beta_k & = \frac{- \| F_k - I \|_F^2 \tr (F_k^2-I) + \tr(F_k -I)\tr ( (F_k^2-I)(F_k-I)) }{ \| F_k -I \|_F^2 \| F_k^2 - I \|_F^2 - \left( \tr ( (F_k^2-I)(F_k-I)) \right)^2} \\
& = \frac{ \| I-F_k\|_F^2 \tr (I-F_k^2) - \tr(I-F_k)\tr ( (I-F_k^2)(I-F_k)) }{ \| I- F_k\|_F^2 \| I- F_k^2\|_F^2 - \left( \tr ( (I-F_k^2)(I-F_k)) \right)^2},
\endaligned
\ee
provided that the denominator $D_k=\| I- F_k\|_F^2 \| I- F_k^2\|_F^2 - \left( \tr ( (I-F_k^2)(I-F_k)) \right)^2\neq 0,$ (while otherwise it is $(\alpha_k,\beta_k)=(0,1)$). For simplicity, denote the numerators of the expressions in \eqref{alfabeta1}, by $A_k$ and $B_k$, respectively.
The expressions \eqref{alfabeta1} are equivalent to \eqref{koef}, and were derived primarily for the purpose of the convergence analysis.

\medskip

The following result provides some basic properties of the sequence $(F_k)_{k=0}^\infty$.

\begin{thm}\label{Osobine}
The following holds for the sequence $(F_{k})_{k=0}^\infty$:
\begin{enumerate}
  \item[\textbf{(a)}] If $D_k\neq 0$, then $\tr(F_{k+1} (I-F_k))=0$ and $\tr(F_{k+1}(I-F_k^2))=0$;
  \item[\textbf{(b)}] If $D_{k-1}\neq 0$, then $\| F_k\|_F^2=\tr(F_k)$ and $\| I-F_k\|_F^2 = \tr(I-F_k)$;
  \item[\textbf{(c)}] If $D_{k-1}\neq 0$, then $\| F_k \|_F^2 + \| I- F_k\|_F^2 = n$;
  \item[\textbf{(d)}] $\|F_k\|_F^2 - \|F_{k+1} \|_F^2 = (\alpha_k + \beta_k-1)\cdot \| I- F_k \|_F^2$ for all $k\in \mathbb{N}$ such that $D_k\neq 0$;
  \item[\textbf{(e)}] If $D_k\neq 0$ for all $k\in\mathbb{N}_0$ then the sequence $(\|F_k\|)_{k=0}^\infty$ is a non-increasing sequence, $(\| I-F_{k}\|_F)_{k=0}^\infty$ is a non-decreasing sequence, $\displaystyle\lim_{k\to \infty} \| F_k \|_F = L$, and $\displaystyle\lim_{k\to \infty} \| I-F_k \|_F = \sqrt{n -L^2}$.
\end{enumerate}
\end{thm}
\begin{proof}

Let $k\in\N$ be arbitrary.

\begin{enumerate}

\item[\textbf{(a)}] From \eqref{smena2} and using fact that $\alpha_k$ and $\beta_k$ 
from \eqref{koef} we get
$$
\begin{aligned}
    \tr(F_{k+1}(I-F_k)) & = \tr( (I-\alpha_k(I-F_k)-\beta_k(I-F_k^2))(I-F_k) \\
    & = \tr(I-F_k) - \alpha_k \tr( (I-F_k)^2) - \beta_k\tr(I-F_k^2) \\
    & = \tr(I-F_k) - \left(\underbrace{\alpha_k \| I-F_k\|_F^2 + \beta_k \tr(I-F_k^2)(I-F_k)}_{=\tr(I-F_k)}\right) = 0, \\
    \tr(F_{k+1}(I-F_k^2)) & = \tr( (I-\alpha_k(I-F_k)-\beta_k(I-F_k^2))(I-F_k^2) \\
    & = \tr(I-F_k^2) - \alpha_k \tr( (I-F_k)(I-F_k^2)) - \beta_k\tr((I-F_k^2)^2) \\
    & = \tr(I-F_k^2) - \left(\underbrace{\alpha_k \tr( (I-F_k)(I-F_k^2)) + \beta_k \| I-F_k^2 \|_F^2}_{=\tr(I-F_k^2)}\right) = 0.
\end{aligned}
$$ 
\item[\textbf{(b)}] From part \textbf{(a)} it immediately follows that 
$$
\begin{aligned}
    \| F_k \|_F^2 & = \tr( F_k (I -\alpha_{k-1} (I-F_{k-1}) -\beta_{k-1} (I-F_{k-1}^2))) \\
    & = \tr(F_k) - \alpha_{k-1} \underbrace{\tr(F_k(I-F_{k-1})}_{=0} - \beta_{k-1} \underbrace{\tr(F_k(I-F_{k-1})^2))}_{=0} =\tr(F_k)
\end{aligned}
$$
and 
$
\| I-F_k\|_F^2 = \tr( (I-F_k)^2) = \tr(I-2F_k + F_k^2) =\tr(I-F_k),
$
provided that $D_{k-1}\neq 0$.

\item[\textbf{(c)}]
From the previous part we have that 
$
\|F_k\|_F^2 + \|I-F_k\|_F^2 = \tr(F_k) + \tr(I-F_k) = \tr(I)= n.
$
\item[\textbf{(d)}] Using the result from part \textbf{(b)} we get:
$$
\begin{aligned}
   \|F_k\|_F^2 - \|F_{k+1}\|_F^2  &= \tr(F_k) - \tr(F_{k+1}) \\
    & = \tr(F_k) - \tr( I - \alpha_k (I-F_k) - \beta_k (I-F_k^2)) \\
    & = \tr( - (I-F_k) + \alpha_k (I-F_k) + \beta_k (I-F_k)) \\
    & = \tr ( (\alpha_k+\beta_k-1) (I-F_k))  = (\alpha_k + \beta_k -1)\cdot \tr(I-F_k) \\&= (\alpha_k + \beta_k -1)\cdot \| I-F_k \|_F^2.
\end{aligned}
$$
\item[\textbf{(e)}]  By construction of the method we have that the sequence $(\|F_k\|_F)_{k=0}^\infty$ is non-increasing, and from part \textbf{(c)} we see that $\|I-F_k\|_F^2 = n - \|F_k\|_F^2$ which implies that $(\|I-F_k\|_F)_{k=0}^\infty$ is non-decreasing, and we have the equality $\displaystyle\lim_{k\to \infty}\| I-F_{k}\|_F=\sqrt{n-L^2}$ where $L=\displaystyle\lim_{k\to \infty} \| F_k \|_F$.\qedhere
\end{enumerate}
\end{proof}

Before we proceed with a more detailed analysis of the convergence of the method \textbf{SSHP2} let us examine the critical case when $D_k=0$ for some $k\in\mathbb{N}_0$.

\begin{thm}\label{trivija} 
Suppose that $D_{k_0}=0$ for some $k_0\in\mathbb{N}_0$. Then $(F_k)_{k=0}^\infty$  converges to an idempotent matrix. Furthermore, if $1\not\in\sigma(F_k)$ for $k\in\mathbb{N}$ the iterative method \textup{\textbf{SSHP2}} converges.
\end{thm}
\begin{proof}
 Let $k_0$ be the first non-negative integer for which $D_k=0$.
By construction, we have that $$\| F_{k_0} \|_F \leqslant \| F_{k_0-1}\|_F \leqslant \dots \leqslant \| F_0 \| \leqslant n\cdot r(F_0).$$
From the Cauchy-Schwarz inequality it follows that $D_{k_0}=0$ if and only if the matrices $I-F_{k_0}$ and $I-F_{k_0}^2$ are linearly  dependent in $\mathbb{R}^{n\times n}$, that is if $I-F_{k_0}^2=c\cdot(I-F_{k_0})$ for some $c\in\mathbb{R}\setminus\lbrace 0 \rbrace$. This implies $$1-\lambda_{i,k_0}^2 = c\cdot(1-\lambda_{i,k_0})\iff (1-\lambda_{i,k_0})(1+\lambda_{i,k_0}-c)=0,\qquad i\in\{1,\dots,n\}.$$
Thus
$
\lambda_{i,k_0} = c-1$ or $\lambda_{i,k} = 1$ for $i\in\{1,\dots, n\}
$.
This means that there exists a (possibly trivial) orthogonal projection $P$ such that 
$
F_{k_0} = (c-1)P + I-P.
$
However, since $D_{k_0-1}\neq 0$, it follows by part \textbf{(b)} of Theorem \ref{Osobine} that $\tr(I-F_{k_0}^2) = \tr(I-F_{k_0})$, which in turn implies that $c=1$, so
$
F_{k_0} = I-P.
$
From the definition of the sequence $(F_k)_{k=0}^\infty$, it follows that all subsequence iterations will also be equal to the projection $I-P$.  
If $1\not \in\sigma(F_k)$ for all $k\in\mathbb{N}$ it follows that $F_{k_0}=0$.

If $D_0=0$ we again reach the conclusion that $\lambda_{i,0}=c-1$ or $\lambda_{i,0}=1$. From the assumption that $\alpha < 2/ \| A \|_2^2$ it follows that $r(F_0)<1$, so it follows that $\lambda_{i,0}=c-1$ where $|c-1|<1$, and subsequently $F_0=(c-1)I$. It follows that $F_k=(c-1)^k I$, and since $|c-1|<1$ we conclude that $F_k$ converges to the zero matrix.
\end{proof}

\begin{thm}\label{Osobine2}
The following holds for the sequence $(F_{k})_{k=0}^\infty$:
\begin{itemize}
    \item[\textbf{(a)}] The sequence $(\|F_k\|)_{k=0}^\infty$ is non-increasing and convergent;
    \item[\textbf{(b)}] For all $k\in \mathbb{N}$ we have $\alpha_k + \beta_k \geqslant1$  and $\displaystyle\lim_{k\to \infty}(\alpha_k+\beta_k)=1$.
\end{itemize}
\end{thm}
\begin{proof}

\begin{itemize}
\item[\textbf{(a)}] This result follows directly from Theorem \ref{Osobine} and the analysis of the case when $D_k=0$.
   
   \item[\textbf{(b)}] From part \textbf{(d)} of Theorem \ref{Osobine} using that $(\|F_k\|_F)_{k=0}^\infty$ is a non-increasing sequence we get $$(\alpha_k + \beta_k -1)\cdot \| I-F_k \|_F^2\geqslant0$$whenever $D_k\neq 0$. Thus, we have $\alpha_k+\beta_k\geqslant 1$ when $D_k\neq 0$. If $D_{k_0}=0$ for some $k_0\in \mathbb{N}_0$ then $\alpha_k+\beta_1=1$ for all natural numbers $k\geqslant k_0$. 
   We conclude that if there exists a non-negative integer $k_0$ such that $D_{k_0}=0$ then the sequence $(\alpha_k+\beta_k)_{k=0}^\infty$ trivially converges to $1$, and if $D_k\neq 0$ for all $k\in\mathbb{N}_0$ it follows follows from part \textbf{(a)} that $$\lim_{k\to\infty}(\alpha_k + \beta_k -1)
   \cdot\| I-F_k \|_F^2=0$$and from part \textbf{(e)} of Theorem \ref{Osobine} the expression $\|I-F_k\|_F^2$ is positive and bounded from below, it follows that $\lim\limits_{k\to\infty}(\alpha_k+\beta_k-1)=0$.\qedhere
\end{itemize}
\end{proof}

\begin{rem}
According to Theorem \ref{trivija}, we mainly study the iterative process \textup{\textbf{SSHP2}} with the additional condition that $D_k\neq 0$ for all $k\in\mathbb{N}_0$ without loss of generality.
\end{rem}

Notice that formulas \eqref{alfabeta1} can be rewritten in a way that will allow us to analyze their behavior in a different way. For simplicity, let $A_k$ denote the numerator of $\alpha_k$, $B_k$ the numerator of $\beta_k$ and $D_k$ their common denominator (and $D_k \geqslant 0$ due to the Cauchy-Schwarz inequality).  Notice that using Theorem \ref{Osobine} $A_k$, $B_k$ and $D_k$ can be further rewritten as 
\begin{align}
    A_k & -\| I-F_k\|_F^2 \,\tr(F_k^3 - F_k^4) \label{Ak2}, \\
    B_k& =  \| I-F_k\, \|_F^2 \tr(F_k^2 - F_k^3) \label{Bk2}, \\
    D_k & = \| I-F_k\|_F^2\, \| F_k - F_k^2 \|_F^2 - (\tr(F_k^2 - F_k^3))^2 \label{Dk2}.
\end{align}
Indeed, from Theorem \ref{Osobine} we get
$$
\begin{aligned}
    D_k& = \| I-F_K^2\|_F^2 \tr(I-F_k) - (\tr(( I-F_k)(I-F_k^2)))^2 \\
    & = \tr( (I-F_k^2)^2) \tr(I-F_k) - ( \tr(I-F_k - F_k^2 +F_k^3))^2 \\
    & = \tr( I-2F_k^2 + F_k^4) \tr(I-F_k) - ( \tr(I-F_k) - \tr(F_k^2- F_k^3))^2 \\
    & = \tr(I-F_k) \tr(I-2F_k + F_k^4) - (\tr(I-F_k))^2 + 2 \tr(I-F_k)\tr(F_k^2-F_k^3) - (\tr(F_k^2 - F_k^3))^2 \\
    & = \tr(I-F_k) \left[  \tr(I-2F_k + F_k^4) - \tr(I-F_k) + 2\tr(F_k^2- F_k^3) \right] - (\tr(F_k^2 - F_k^3))^2 \\
    & = \tr(I-F_k) \tr(F_k^4 + F_k-2F_k^3 ) - (\tr(F_k^2 - F_k^3))^2 \\
    & = \tr(I-F_k) \tr(F_k^4 + F_k^2-2F_k^3 )- (\tr(F_k^2 - F_k^3))^2 \\
    & = \tr(I-F_k) \tr( (F_k-F_k^2)^2) -  (\tr(F_k^2 - F_k^3))^2  = \| I-F_k\|_F^2 \| F_k - F_k^2 \|_F^2 - (\tr(F_k^2 - F_k^3))^2, \\ \\
    B_k & =  (\tr(I-F_k))^2 - \tr(I-F_k) \tr((I-F_k)(I-F_k^2)) \\
    & = \tr(I-F_k) \left[ \tr(I-F_K) -  \tr((I-F_k)(I-F_k^2)) \right] \\
    & = \tr(I-F_k) \tr( (I-F_k) - (I-F_k)(I-F_k^2)) \\
    & = \tr(I-F_k) \tr( (I-F_k) ( I- (I-F_k^2))) = \tr(I-F_k) \tr( (I-F_k)F_k^2) = \| I-F_k \|_F^2 \tr(F_k^2 - F_k^3),\\ \\
    A_k & = \tr(I-F_k) \| I-F_k^2\|_F^2 - \tr(I-F_k) \tr( (I-F_k)(I-F_k^2)) \\
    & = \tr(I-F_k) \left( \tr( I-F_k^2)^2 - \tr( (I-F_k)(I-F_k^2)) \right) \\
    & = \tr(I-F_k) \tr\left( (I-F_k^2)^2 - (I-F_k)(I-F_k^2)   \right) \\
    & = \tr(I-F_k) \tr\left( (I-F_k^2)(F_k - F_k^2) \right) \\
    &= \tr(I-F_k) \tr \left( F_k - F_k^2 - F_k^3 + F_k^4 \right)  = -\tr(I-F_k) \tr(F_k^3- F_k^4)  = -\| I-F_k\|_F^2 \tr(F_k^3 - F_k^4).
\end{aligned}
$$
This gives us the following formulas for $\alpha_k$ and $\beta_k$ when $D_k\neq 0$:
\begin{equation}\label{alfa2}
    \alpha_k = \frac{A_k}{D_k} = \frac{-\| I-F_k\|_F^2 \tr(F_k^3 - F_k^4)}{\| I-F_k\|_F^2 \| F_k - F_k^2 \|_F^2 - (\tr(F_k^2 - F_k^3))^2},
\end{equation}
\begin{equation}\label{beta2}
\beta_k = \frac{B_k}{D_k} = \frac{\| I-F_k \|_F^2 \tr(F_k^2 - F_k^3)}{\| I-F_k\|_F^2 \| F_k - F_k^2 \|_F^2 - (\tr(F_k^2 - F_k^3))^2}.    
\end{equation}
The formulas \eqref{alfa2} and \eqref{beta2} can be used to obtain useful bounds for $\alpha_k$ and $\beta_k$.
\begin{lem}\label{Lgranice}
 If $D_k\neq 0$ for $k\in\mathbb{N}_0$ and $\lambda_{i,k} <1$ for $i\in\{1,\dots,n\}$ then 
 \begin{equation}\label{granica1}
     \beta_k \geqslant \frac{1}{1+r(F_k)}, \qquad |\alpha_k| \leqslant r(F_k) \beta_k.
 \end{equation}
Additionally, if $r(F_k)<1$ we have the inequalities 
\begin{equation*}\label{granica2}
     \frac{1}{1+r(F_k)} \leqslant \beta_k \leqslant \frac{1}{(1-r(F_k))^2}, \qquad |\alpha_k| \leqslant \frac{r(F_k)}{(1-r(F_k))^2}.
 \end{equation*}
\end{lem}
\begin{proof}
    First, suppose that $D_k\neq 0$ for some $k\in\mathbb{N}$ and that $\lambda_{i,k} <1$ for $i=1,\dots,n$. We have that $\beta_k$ must be non-negative since $\tr(F_k^2-F_k^3)=\sum_{i=1}^n \lambda_{i,k}^2 (1-\lambda_{i,k})$ is non-negative.
    From
    $$
    \begin{aligned}
        \left|-\tr(F_k^3-F_k^4)\right|& = |\tr(F_k^3-F_k^4)| = \left\vert \sum_{i=1}^n \lambda_{i,k}^3 (1-\lambda_{i,k}) \right\vert \\&\leqslant \sum_{i=1}^n |\lambda_{i,k}^3 |\cdot(1-\lambda_{i,k})
    \leqslant r(F_k) \sum_{i=1}^n \lambda_{i,k}^2 (1-\lambda_{i,k}) = r(F_k)\cdot \tr(F_k^2-F_k^3)
    \end{aligned}
    $$
    we obtain the estimate $|A_k| \leqslant r(F_k) B_k$, from which the inequality $|\alpha_k|\leqslant r(F_k)\beta_k$ immediately follows.
Next, let us write $D_k$ and $B_k$ as:
    $$
    \begin{aligned}
      D_k& = \| I-F_k\|_F^2 \|F_k-F_k^2\|_F^2 - (\tr(F_k^2 - F_k^3))^2 = \tr(I-F_k) \tr ( (F_k-F_k^2)^2) -  (\tr(F_k^2 - F_k^3))^2 \\
      & = \sum_{i=1}^n (1-\lambda_{i,k}) \sum_{j=1}^n \lambda_{j,k}^2 (1-\lambda_{j,k})^2  - \sum_{i=1}^n \lambda_{i,k}^2 (1-\lambda_{i,k}) \sum_{j=1}^n \lambda_{j,k}^2(1-\lambda_{j,k})  \\
      & = \sum_{i=1}^n \sum_{j=1}^n (1-\lambda_{i,k})(1-\lambda_{j,k})^2  \lambda_{j,k}^2 - \sum_{i=1}^n \sum_{j=1}^n (1-\lambda_{i,k})(1-\lambda_{j,k}) \lambda_{i,k}^2 \lambda_{j,k}^2 \\
      & = \sum_{i=1}^n \sum_{j=1}^n (1-\lambda_{i.k})(1-\lambda_{j,k}) \lambda_{j,k}^2 (1-\lambda_{j,k}-\lambda_{i,k}^2),\\[4pt] \medskip
      B_k& = \| I-F_k \|_F^2 \tr(F_k^2 - F_k^3) = \tr(I-F_k)\tr(F_k^2-F_k^3) \\
      & = \sum_{i=1}^n (1-\lambda_{i,k}) \sum_{j=1}^n \lambda_{j.k}^2 (1-\lambda_{j,k}) = \sum_{i=1}^n \sum_{j=1}^n (1-\lambda_{i,k})(1-\lambda_{j,k})\lambda_{j,k}^2. 
    \end{aligned}
    $$
    Since $D_k>0$ we have that 
    $$
    \begin{aligned}
    0 < D_k  & = \sum_{i=1}^n \sum_{j=1}^n (1-\lambda_{i.k})(1-\lambda_{j,k}) \lambda_{j,k}^2 (1-\lambda_{j,k}-\lambda_{i,k}^2) \leqslant \sum_{i=1}^n \sum_{j=1}^n (1-\lambda_{i.k})(1-\lambda_{j,k}) \lambda_{j,k}^2 (1-\lambda_{j,k}) \\
    & \leqslant (1+r(F_k)) \sum_{i=1}^n \sum_{j=1}^n (1-\lambda_{j,k})(1-\lambda_{i,k})\lambda_{j,k}^2 = (1+r(F_k)) B_k,
    \end{aligned}
    $$
    which in turn implies that $\beta_k \geqslant (1+r(F_k))^{-1}$. 

\medskip

    Now assume that $r(F_k)<1$. In order to show that $\beta_k \leqslant (1-r(F_k)^{-2}$, let us get back to the second formula in \eqref{alfabeta1} and rewrite $B_k$ and $D_k$ in a different form. First, we rewrite $B_k$ as
$$
\begin{aligned}
    B_k & = \| I-F_k\|_F^2 \tr (I-F_k^2) - \tr(I-F_k)\tr ( (I-F_k^2)(I-F_k))  \\
     & = \sum_{i=1}^n \sum_{j=1}^n \left[  (1-\lambda_{i,k}^2)(1-\lambda_{j,k})^2 - (1-\lambda_{i,k})(1-\lambda_{j,k})(1-\lambda_{j,k}^2)  \right] \\
     & = \sum_{i=1}^n \sum_{j=1}^n (1-\lambda_{i,k})(1-\lambda_{j,k}) \left[ (1+\lambda_{i,k}) (1-\lambda_{j,k}) - (1-\lambda_{j,k}^2) \right] \\
     & = \sum_{i=1}^n \sum_{j=1}^n (1-\lambda_{i,k})(1-\lambda_{j,k})^2 (1+\lambda_{i,k} -1 - \lambda_{j,k})  = \sum_{i=1}^n \sum_{j=1}^n (1-\lambda_{i,k})(1-\lambda_{j,k})^2 (\lambda_{i,k} - \lambda_{j,k}).
\end{aligned}
$$
Denote by $t_{i,j}$ the $(i,j)$-th element of the sum above. Then
$
B_k = \sum_{1\leqslant i < j \leqslant n} (t_{i,j}+t_{j,i}).
$
As 
$$
t_{j,i} = (1-\lambda_{j,k})(1-\lambda_{i,k})^2 (\lambda_{j,k} - \lambda_{i,k})  = - (1-\lambda_{i,k})^2 (1-\lambda_{j,k}) ( \lambda_{i.k} - \lambda_{j,k}),
$$
it stands that 
$$
\begin{aligned} 
t_{i,j} + t_{j,i} & = (1-\lambda_{i,k})(1-\lambda_{j,k})^2 (\lambda_{i,k} - \lambda_{j,k}) - (1-\lambda_{i,k})^2 (1-\lambda_{j,k}) ( \lambda_{i.k} - \lambda_{j,k}) \\
& = (1-\lambda_{i,k})(1-\lambda_{j,k}) ( \lambda_{i,k} - \lambda_{j,k}) (1- \lambda_{j,k} - (1- \lambda_{i,k})) \\
& = (1-\lambda_{i,k})(1-\lambda_{j,k})  (\lambda_{i,k}-\lambda_{j,k})^2 .
\end{aligned}
$$
The previous calculations directly imply
\bb\label{Bkdouble}
\aligned
B_k & = \sum_{1\leqslant i < j \leqslant n } (1-\lambda_{i,k})(1-\lambda_{j,k})  (\lambda_{i,k}-\lambda_{j,k})^2 = \frac{1}{2} \sum_{i=1}^n \sum_{j=1}^n (1-\lambda_{i,k})(1-\lambda_{j,k})  (\lambda_{i,k}-\lambda_{j,k})^2.
\endaligned
\ee
Finally, one can express $D_k$ as:
$$
\begin{aligned}
    D_k  &= \| I- F_k\|_F^2\cdot \| I- F_k^2\|_F^2 - \left( \tr ( (I-F_k^2)(I-F_k)) \right)^2 \\
    & = \sum_{i=1}^n \sum_{j=1}^n \left[  (1-\lambda_{i,k})^2  (1-\lambda_{j,k}^2)^2 - (1-\lambda_{i,k})(1-\lambda_{i,k}^2) (1-\lambda_{j,k})(1-\lambda_{j,k}^2) \right] \\
    & = \sum_{i=1}^n \sum_{j=1}^n (1-\lambda_{i,k})(1-\lambda_{j,k}^2) \left[ (1-\lambda_{i,k})(1-\lambda_{j,k}^2) - (1-\lambda_{i,k}^2)(1-\lambda_{j,k})  \right] \\
    & = \sum_{i=1}^n \sum_{j=1}^n (1-\lambda_{i,k})(1-\lambda_{j,k}^2)\left[ (1-\lambda_{i,k})(1+\lambda_{j,k})(1-\lambda_{j,k}) - (1-\lambda_{i,k})(1+\lambda_{i,k})(1-\lambda_{j,k}) \right] \\
    & = \sum_{i=1}^n \sum_{j=1}^n (1-\lambda_{i,k})(1-\lambda_{j,k}^2)(1-\lambda_{i,k})(1-\lambda_{j,k}) ( 1 + \lambda_{j,k} -1 - \lambda_{i,k})  \\&=\sum_{i=1}^n \sum_{j=1}^n (1-\lambda_{i,k})^2(1-\lambda_{j,k}^2)(1+\lambda_{j,k})(\lambda_{j,k}- \lambda_{i,k}).
\end{aligned}
$$
Denote by $u_{i,j}$ the $(i,j)$-th summand in the sum above. Then
$D_k  = \sum_{1\leqslant i < j \leqslant n} ( u_{i,j} + u_{j,i}).$
Since 
$$
\begin{aligned}
u_{j,i} &=(1-\lambda_{j,k})^2(1-\lambda_{i,k}^2)(1+\lambda_{i,k})(\lambda_{i,k}- \lambda_{j,k}) = - (1-\lambda_{i,k})^2(1-\lambda_{j,k})^2 (1+\lambda_{i,k})(\lambda_{j,k}-\lambda_{i,k}),
\end{aligned}
$$
it follows that
$$
\begin{aligned}
u_{i,j} + u_{j,i} & = (1-\lambda_{i,k})^2(1-\lambda_{j,k}^2)(1+\lambda_{j,k})(\lambda_{j,k}- \lambda_{i,k})- (1-\lambda_{i,k})^2(1-\lambda_{j,k})^2 (1+\lambda_{i,k})(\lambda_{j,k}-\lambda_{i,k})\\
& = (1-\lambda_{i,k})^2 (1-\lambda_{j,k})^2 (\lambda_{j,k}-\lambda_{i,k} (1+\lambda_{j,k} - 1 - \lambda_{i,k}) \\
& = (1-\lambda_{i,k})^2 (1-\lambda_{j,k})^2 (\lambda_{j,k}-\lambda_{i,k})^2.
\end{aligned}
$$
and therefore
\bb \label{Dkdouble}
\aligned
D_k & =\! \!\sum_{1\leqslant i < j \leqslant n} (1-\lambda_{i,k})^2 (1-\lambda_{j,k})^2 (\lambda_{j,k}-\lambda_{i,k})^2=\frac{1}{2} \sum_{i=1}^n \sum_{j=1}^n (1-\lambda_{i,k})^2 (1-\lambda_{j,k})^2 \left(\lambda_{j,k}-\lambda_{i,k} \right)^2.
\endaligned
\ee
Combining the formulas \eqref{Bkdouble} and \eqref{Dkdouble} we obtain
\begin{align}\beta_k & = \frac{\displaystyle\sum_{1\leqslant i < j \leqslant n } (1-\lambda_{i,k})(1-\lambda_{j,k})  (\lambda_{i,k}-\lambda_{j,k})^2}{\displaystyle\sum_{1\leqslant i < j \leqslant n} (1-\lambda_{i,k})^2 (1-\lambda_{j,k})^2 (\lambda_{j,k}-\lambda_{i,k})^2} =\frac{\displaystyle\sum_{i=1}^n \sum_{j=1}^n (1-\lambda_{i,k})(1-\lambda_{j,k})  (\lambda_{i,k}-\lambda_{j,k})^2}{\displaystyle\sum_{i=1}^n \sum_{j=1}^n (1-\lambda_{i,k})^2 (1-\lambda_{j,k})^2 (\lambda_{j,k}-\lambda_{i,k})^2}.\label{betadouble} 
\end{align}
Finally, from \eqref{betadouble} we see that
$$
\beta_k  = \frac{\displaystyle\sum_{i=1}^n \sum_{j=1}^n (1-\lambda_{i,k})(1-\lambda_{j,k})  (\lambda_{i,k}-\lambda_{j,k})^2}{\displaystyle\sum_{i=1}^n \sum_{j=1}^n (1-\lambda_{i,k})^2 (1-\lambda_{j,k})^2 (\lambda_{j,k}-\lambda_{i,k})^2} = \frac{\displaystyle\sum_{i=1}^n \sum_{j=1}^n \dfrac{(1-\lambda_{i,k})^2(1-\lambda_{j,k})^2  (\lambda_{i,k}-\lambda_{j,k})^2} {(1-\lambda_{i,k})(1-\lambda_{j,k})}}{\displaystyle\sum_{i=1}^n \sum_{j=1}^n (1-\lambda_{i,k})^2 (1-\lambda_{j,k})^2 (\lambda_{j,k}-\lambda_{i,k})^2}.
$$
It allows us to form the following bound
$$
\beta_k \leqslant \max\left\{ \frac{1}{(1-\lambda_{i,k})(1-\lambda_{j,k})}:\lambda_{i,k},\lambda_{j,k} \in \sigma(F_k)\right\} 
$$
which directly implies $\beta_k \leqslant (1-r(F_k))^{-2}$.
\end{proof}

\begin{thm}\label{PrvaT} 
If the method \textup{\textbf{SSHP2}} converges, then $$\displaystyle\lim_{k\to \infty} \alpha_k=0,\qquad\displaystyle\lim_{k\to \infty} \beta_k = 1.$$
\end{thm}
\begin{proof}
If the methods \textbf{SSHP2} converges, then $\displaystyle\lim_{k\to \infty} \|F_k\|_F=0$, so it follows that $r(F_k)<1$ for sufficiently large $k\in \mathbb{N}$.  Notice that if $D_k=0$ for some $k\in\mathbb{N}_0$ we will trivially have that $\lim\limits_{k\to\infty}\alpha_k= 0$ and $\lim\limits_{k\to\infty}\beta_k= 1$. 
We proceed in the case when $D_k\neq 0$ for $k\in\mathbb{N}_0$. 
Applying Lemma \ref{Lgranice} and $\lim\limits_{k\to\infty}r(F_k)=0$ we get that $\lim\limits_{k\to \infty} \beta_k=1$ and $\lim\limits_{k\to\infty} \alpha_k=0$. 
\end{proof}

We now formulate the main result regarding the convergence of \textbf{SSHP2} method.

\begin{thm}\label{Glavna1}
    If $1\not\in \sigma(F_k)$ for $k\in\mathbb{N}_0$, and  there exists a non-negative integer $l$ and a positive constant $m$ such that $\beta_k>m$ for all $k\geq l$, the iterative process \textup{\textbf{SSHP2}} converges.
\end{thm}
\begin{proof}
Let us notice that if $\lambda_{i,k_0}=1$ for some $i\in\{1,\dots, n\}$ and $k_0\in\mathbb{N}$, then $\lambda_{i,k}=1$ for all $k\geqslant k_0$. So the condition that $1\not \in \sigma(F_k)$ for all $k\in\mathbb{N}_0$ is indeed necessary. Without loss of generality, we proceed under the assumption that $D_k\neq 0$ for all $k\in\mathbb{N}_0$ as this case was covered in Theorem \ref{trivija}.

\medskip

We first show that under these assumptions, the sequence $(F_k)_{k=0}^\infty$ converges to a projection matrix. Then we will show that the only possible limit of the sequence $(F_k)_{k=0}^\infty$ is the zero matrix.

\medskip

Recall that $A_k$, $B_k$ and $D_k$ denote the numerators of $\alpha_k$ and $\beta_k$, and their common denominator from \eqref{alfabeta1}. The numerator of $\alpha_k+\beta_k-1$ (which is equal to $A_k+B_k-D_k$) can be written using formulas \eqref{Ak2},\,\eqref{Bk2} and \eqref{Dk2} as:
$$
\begin{aligned}
    A_k+B_k - D_k & = \|I-F_k\|_F^2 (\tr(F_k^2-F_k^3) - \tr(F_k^3 - F_k^4)) - \|I-F_k^2\|_F^2 \|F_k-F_k^2\|_F^2  + (\tr(F_k^2-F_k^3))^2 \\
    & = \| I -F_k\|_F^2 \tr(F_k^2 - 2F_k^3 +F_k^4) - \|I-F_k^2\|_F^2 \|F_k-F_k^2\|_F^2  + (\tr(F_k^2-F_k^3))^2\\& = \left( \tr(F_k^2- F_k^3) \right)^2.
\end{aligned}
$$
By part \textbf{(b)} of Theorem \ref{Osobine2}, it follows that $\lim\limits_{k\to\infty}(A_k+B_k-D_k)= 0$, so 
\begin{equation}\label{tragFk}
    \displaystyle\lim_{k\to \infty} \tr(F_k^2-F_k^3) = 0.
\end{equation}
In \eqref{beta2} from \eqref{tragFk}, we see that the numerator of $\beta_k$ always converges to $0$ as $k\to \infty$. Since all values of $\beta_k$ are bounded from bellow by a positive number all for $k\geq l$, it follows that $D_k$ converges to $0$ as well.  This implies that 
$\displaystyle\lim_{k\to \infty} \| I-F_k \|_F^2\cdot\| F_k-F_k^2\|_F^2 =0.$ 
Since $(\| I-F_k \|_F)_{k=0}^\infty$ is a bounded non-decreasing sequence and non-zero by construction, it follows that 
$\displaystyle\lim_{k\to \infty} \| F_k-F_k^2\|_F^2 =0,$
which allows us to conclude that $F_k$ will converge to a projection matrix $P$. Indeed, $\displaystyle\lim_{k\to \infty} \| F_k-F_k^2\|_F^2 =0$ is equivalent to
$$
\lim_{k\to\infty}\sum_{i=1}^n\lambda_{i,k}^2(1-\lambda_{i,k})^2=0,$$
and since all summands are non-negative it follows that each of the summands converges to $0$, so $\displaystyle\lim_{k\to \infty} \lambda_{i,k}^2(1-\lambda_{i,k})^2=0$ for $i=1,\dots,n$. In turn it follows that $\displaystyle\lim_{k\to\infty} \lambda_{i,k}$ exists, and it is equal to $0$ or $1$. 

\medskip

Suppose now that the projection matrix $P$ that $(F_k)_{k=0}^\infty$ converges to is not the zero matrix, so there exists an index $i\in \lbrace 1,\dots, n\rbrace$ such that $\displaystyle\lim_{k\to \infty} \lambda_{i,k} =1$.  
Subtracting $1$ from both sides of \eqref{vezab0b1} gives us 
$$ 
\begin{aligned} 
\lambda_{i,k+1} -1 & = -\alpha_k - \beta_k + \alpha_k \lambda_{i,k} + \beta \lambda_{i,k}^2  \\&=\alpha_k (\lambda_{i,k}-1) + \beta_k(\lambda_{i,k}^2 - 1) = (\lambda_{i,k}-1) ( \alpha_k + \beta_k (\lambda_{i,k}+1)).
\end{aligned}
$$
Since $1\not \in \sigma(F_k)$, this equality can be rewritten as 
\begin{equation*}\label{kolicnik}
     \frac{\lambda_{i,k+1}-1}{\lambda_{i,k}-1} = \alpha_k + \beta_k (\lambda_{i,k}+1) = \alpha_k + \beta_k + \beta_k\lambda_{i,k}.
\end{equation*}
Let $\varepsilon >0$ be an arbitrary real number. Since  $\displaystyle\lim_{k\to \infty} \lambda_{i,k}=1$ and all accumulation points of $\beta_k$ are positive, there exists a positive integer $k_0$ such that
$
1- \varepsilon < \lambda_{i,k} < 1+ \varepsilon$ and $\beta_k >0
$
for all $k\geqslant k_0$. For all such $\alpha_k$, $\beta_k$ and $\lambda_{i,k}$, using part \textbf{(d)} from Theorem \ref{Osobine} we get
$\alpha_k + \beta_k + \beta_k\lambda_{i,k} > \alpha_k + \beta_k \geqslant 1$. 
So it follows that 
\begin{equation}\label{kolicnik2}
     \frac{\lambda_{i,k+1}-1}{\lambda_{i,k}-1} >1,\qquad k\geqslant k_0.
\end{equation} 
If $\lambda_{i,k}>1$ then the inequality \eqref{kolicnik2} implies that $\lambda_{i,k+1} > \lambda_{i,k} >1$. Now replacing $k$ in \eqref{kolicnik2} by $k+1,k+2, \ldots$, we see that the sequence $(\lambda_{i,k})_{k \geq k_0}$ is an increasing sequence, so $\displaystyle\lim_{k\to \infty} \lambda_{i,k} \geqslant \lambda_{i,k_0} >1$. This is contradiction with the assumption that $\displaystyle\lim\limits_{k\to\infty} \lambda_{i,k} =1$. 
Similarly, if $\lambda_{i,k}<1$ then the inequality \eqref{kolicnik2} implies that $\lambda_{i,k+1} < \lambda_{i,k} <1$. In this case replacing $k$ in \eqref{kolicnik2} by $k+1,k+2, \ldots$, leads us to the conclusion that $(\lambda_{i,k})_{k= k_0}^\infty$ is a decreasing sequence, so $\displaystyle\lim_{k\to \infty} \lambda_{i,k} \leqslant \lambda_{i,k_0} <1$, contradicting the assumption that $\displaystyle\lim_{k\to \infty} \lambda_{i,k} =1$.

\medskip

Therefore $P$ must be a zero matrix, which completes the proof.
\end{proof}

The following result follows directly from Theorem \ref{Glavna1}.
\begin{cor}\label{DrugaT} 
If $\displaystyle\lim_{k\to \infty} \alpha_k=0$,  $\displaystyle\lim_{k\to \infty} \beta_k = 1$ and $1\not \in \sigma(F_k)$ for all $k\in \mathbb{N}_0$ then the method \textup{\textbf{SSHP2}} converges.
\end{cor}

Combining the conclusion from Theorem \ref{PrvaT} and Corollary \ref{DrugaT}, we get the next result.
\begin{thm}\label{OsnovnaT}
The iterative methods \textup{\textbf{SSHP2}} converges if and only if $1\not\in \sigma(F_k)$ for all $k\in \mathbb{N}_0$, $\displaystyle\lim_{k\to \infty} \alpha_k=0$ and $\displaystyle\lim_{k\to \infty} \beta_k=1$.     
\end{thm}
Let us provide some auxiliary results about sufficient conditions under which \textup{\textbf{SSHP2}} converges. 
\begin{thm}\label{PomocnaTKonv}
    If $\lambda_{i,0}<1$ for $i \in \{1,2,\ldots,n\}$, and 
    \begin{equation}\label{uslovBkRk}
        \beta_k \leqslant \frac{1}{r(F_k)}
    \end{equation} for all $k\in \mathbb{N}_0$, then the iterative method \textup{\textbf{SSHP2}} converges.
\end{thm}
\begin{proof}
    The first step is to show $\lambda_{i,k}<1$ for all $i \in \{1,\ldots,n\}$ and $k\in\mathbb{N}_0$, using mathematical induction over $k$. The base case $k=0$ is satisfied directly from the assumptions of the Theorem. Assume that for some non-negative integer $k$ we have $\lambda_{i,k}<1$ for $i \in \{1,2,\ldots,n\}$. From \eqref{vezab0b1} we see that 
    \bb
    \lambda_{i,k+1} -1 = (\lambda_{i,k}-1)(\alpha_k + \beta_k + \beta_k \lambda_{i,k}).
    \label{eq:lambdakp1m1}
    \ee
    From the assumption \eqref{uslovBkRk}, one concludes $\beta_k \lambda_{i,k} \geqslant -1$
    %
    for all $\lambda_{i,k}\in \sigma(F_k)$, which in turn implies $1+\beta_k \lambda_{i,k}\geqslant 0$. 
    Furthermore, part \textup{\textbf{(b)}} of Theorem \ref{Osobine2} implies 
    $
    \alpha_k + \beta_k + \beta_k \lambda_{i,k} \geqslant 1+\beta_k \lambda_{i,k} \geqslant 0.
    $
    Now from \eqref{eq:lambdakp1m1} and the induction hypothesis (that is $\lambda_{i,k}-1<0$)
    one obtains $\lambda_{i,k+1}<1$ for $i \in \{1,\ldots,n\}$. 
    This completes the proof by induction.

    \medskip

    Therefore $\beta_k \geqslant 0$ for all $k\in\mathbb{N}_0$, and  $\beta_{k_0}=0$ if and only if $\lambda_{i,k_0}=0$ for some $k_0 \in \mathbb{N}$. Indeed, from the fact that $\lambda_{i,k}<1$ for $i\in\{1,\dots ,n\}$ it follows that all $
    \tr(F_{k}^2-F_{k}^3) = \sum_{i=1}^n \lambda_{i,k}^2(1-\lambda_{i,k})$ is non-negative. Furthermore, since $\lambda_{i,k}^2(1-\lambda_{i,k})\geqslant 0$ we see that $\beta_k=0$ if and only if $\lambda_{i,k}^2(1-\lambda_{i,k})= 0$, where $i=\overline{1, \dots, n}$, which together with the fact that $\lambda_{i,k}<1$ implies that $\lambda_{i,k}=0$ for $i\in\{1,\dots , n\}$.

    Since $\lambda_{i,k}<1$ directly implies $1\not\in\sigma(F_k)$ for all $k\in\mathbb{N}_0$, Theorem \ref{Glavna1} guarantees the convergence of \textup{\textbf{SSHP2}}.
\end{proof}

\begin{thm}\label{PomocnaTKonv2}
    If $r(F_0)<1$ and  
    \begin{equation}\label{uslovBkRk2}
        \beta_k \leqslant \frac{4(r(F_k)+1)}{(r(F_k)+2)^2}
    \end{equation}
    for $k\in \mathbb{N}_0$,  then the iterative method \textup{\textbf{SSHP2}} converges.
\end{thm}
\begin{proof}
    Let us show that if $r(F_k)<1$ for some $k\in \mathbb{N}_0$ the condition \eqref{uslovBkRk2} is satisfied we have that  $r(F_{k+1})\leqslant r(F_k)$. Observe the function 
    $$
    f_k(\lambda) = \beta_k \lambda^2 + \alpha_k \lambda + 1 - (\alpha_k+\beta_k),\qquad \lambda\in\mathbb{R}.
    $$
    We will show that $-r(F_k)\leqslant f_k(\lambda)\leqslant r(F_k)$ for $\lambda \in [-r(F_k),r(F_k)]$. Since $r(F_k)<1$ from \eqref{granica1} we have that $\beta_k >0$, so $f_k(\lambda)$ is a convex quadratic function. Therefore it is sufficient to establish that $f_k(r(F_k))\leqslant r(F_k)$, $f_k(-r(F_k))\leqslant r(F_k)$, and $f_k\left( - \frac{\alpha_k}{2\beta_k}\right) \geqslant -r(F_k)$. 
    Using the inequality from part \textbf{(b)} from Theorem \ref{Osobine2} we have that 
    \begin{align*}
    f_k(r(F_k)) &= \beta_k r(F_k)^2 + \alpha_k r(F_k) + 1 - (\alpha_k+\beta_k) \\
    &< (r(F_k)-1) (\alpha_k + \beta_k) + 1 \leqslant r(F_k)-1 +1 = r(F_k)
    \end{align*}
    Next, we have 
    \begin{align*}
    f_k(-r(F_k)) &= \beta_k r(F_k)^2 - \alpha_k r(F_k) + 1 - (\alpha_k+\beta_k) \\
    &= (r(F_k)+1) \left(\beta_k(r(F_k)-1) - \alpha_k\right) +1,
    \end{align*}
    Notice that it is easily verified that $\frac{4(r(F_k)+1)}{(r(F_k)+2)^2}< \frac{2}{1+r(F_k)}$ so if \eqref{uslovBkRk2} is satisfied then 
    $$
    \beta_k \leqslant \frac{2}{1+r(F_k)} = \frac{r(F_k)-1}{r(F_k)(r(F_k)+1)} + \frac{1}{r(F_k)}.
    $$
    The last inequality is equivalent to
    \begin{equation}\label{T310pom3}
        (r(F_k)+1)(\beta_k r(F_k)-1) \leqslant r(F_k)-1.
    \end{equation}
    Applying part \textbf{(b)} from Theorem \ref{Osobine2} again to the left hand side of \eqref{T310pom3} one finds
    $$
    \begin{aligned}
    r(F_k)-1  &\geqslant (r(F_k)+1)(\beta_k r(F_k)-1) \\
     &\geqslant  (r(F_k)+1)(\beta_k r(F_k)-(\alpha_k+\beta_k)) \\
     &= (r(F_k)+1) \left(\beta_k(r(F_k)-1) - \alpha_k\right)
    \end{aligned}
    $$
    which allows us to conclude that $f_k(-r(F_k)) \leqslant r(F_k)$.
    Using Lemma \ref{Lgranice} we have that 
\begin{equation}\label{T310pom2}
      \alpha_k^2 + 4\alpha_k \beta_k + 4 \beta_k^2 - 4 \beta_k \leqslant (r(F_k)^2+4r(F_k)+4) \beta_k^2 - 4 \beta_k.  
    \end{equation}    
    The inequality \eqref{uslovBkRk2} is equivalent with 
    $$
    (r(F_k)^2+4r(F_k)+4)\beta_k-4\leqslant 4r(F_k).
    $$
    After multiplying both sides of the previous inequality with $\beta_k$ and using
    \eqref{T310pom2}, we obtain 
    \begin{equation}\label{T310pom}
      \alpha_k^2 + 4\alpha_k \beta_k + 4 \beta_k^2 - 4 \beta_k \leqslant 4\beta_k r(F_k).  
    \end{equation}
    Dividing both sides of \eqref{T310pom} with $4\beta_k > 0$, one gets
    $$
    \frac{\alpha_k^2}{4\beta_k} + \alpha_k+\beta_k-1\leqslant r(F_k)
    $$
    and hence
    \begin{align*}
    f_k\left( - \frac{\alpha_k}{2\beta_k} \right) &= - \frac{\alpha_k^2}{4\beta_k} + 1 - (\alpha_k+\beta_k) = - \left( \frac{\alpha_k^2}{4\beta_k} + \alpha_k+\beta_k-1 \right) \geqslant - r(F_k).
    \end{align*}
    We have thus shown that if \eqref{uslovBkRk2} is satisfied, then $-r(F_k) \leqslant f_k(\lambda)\leqslant r(F_k)$, which in turn implies that  $r(F_{k+1})\leqslant r(F_k)$ for all $k\in \mathbb{N}_0$. We have 
    $\beta_k \geqslant\frac{1}{1+r(F_k)} \geqslant \frac{1}{1+r(F_0)}$
    for all $k\in \mathbb{N}_0$, so by applying Theorem \ref{Glavna1} we conclude that the iterative method \textup{\textbf{SSHP2}} converges.
\end{proof}

Theorem \ref{PomocnaTKonv} and Theorem \ref{PomocnaTKonv2} show that unlike the classic Schultz method the iterative method \textup{\textbf{SSHP2}} does not guarantee that the residual matrices have a decreasing spectral radius. However this is not surprising since  the iterative method \textup{\textbf{OPM}} from \cite{OPM} exhibits similar behavior.

We can briefly compare the efficiency of the iterative method \textup{\textbf{OPM}} introduced in \cite{OPM} with the iterative method \textup{\textbf{SSHP2}}. From \cite[Lemma 2]{OPM} we know that the norms of the residuals when \textup{\textbf{OPM}} is applied satisfy
$$
\| F_k \|_F^2 - \| F_{k+1} \|_F^2  = \frac{\left( \tr(F_k^2 - F_k^3)\right)^2}{\|F_k - F_k^2\|_F^2}.
$$
This gain between iterations is smaller than the gain in \textup{\textbf{SSHP2}}. Indeed, from part \textup{\textbf{(d)}} of Theorem \ref{Osobine} we know that 
$
\| F_k\|_F^2 - \| F_{k+1} \|_F^2 = (\alpha_k+\beta_k-1)\| I-F_k\|_F^2.
$
Using the equalities \eqref{Ak2},\,\eqref{Bk2} and \eqref{Dk2} we get
$$
\begin{aligned}
    \| F_k\|_F^2 - \| F_{k+1} \|_F^2 & = (\alpha_k+\beta_k-1)\| I-F_k\|_F^2  = \frac{\left( \tr(F_k^2 - F_k^3)\right)^2}{\| I-F_k\|_F^2 \| F_k - F_k^2 \|_F^2 - \left( \tr(F_k^2 - F_k^3)\right)^2} \cdot\| I-F_k\|_F^2  \\
    & \geqslant \frac{\left( \tr(F_k^2 - F_k^3)\right)^2}{\| I-F_k\|_F^2 \| F_k - F_k^2 \|_F^2}\cdot \| I-F_k\|_F^2  = \frac{\left( \tr(F_k^2 - F_k^3)\right)^2}{\|F_k - F_k^2\|_F^2},
\end{aligned}
$$
which shows that if an iteration of \textup{\textbf{OPM}} and \textup{\textbf{SSHP2}} are performed on the same matrix, the gain in \textup{\textbf{SSHP2}} is higher.

\bigskip

As the last part of the convergence analysis of \textup{\textbf{SSHP2}}, we examine how small perturbations of the sequence $\seqn{X_k}$ affect the method, under the assumption that $\seqn{F_k}$ is not convergenct (that is it does not converge to 0). 
\begin{dfn}
The perturbation of the sequence $\seqn{X_k}$ defined by \eqref{algor0} using the perturbation matrices $\Delta_1,\Delta_2\ldots,\Delta_m$ starting from the index $k_0$, is the sequence $\seqn{{\tilde X}_k}$ defined by:
$$
\aligned
\tilde X_k &=X_k \textnormal{ for } k\in\{0,1,\ldots,k_0\},\\
\tilde X_{k,p} &= \tilde X_{k} + \Delta_{k-k_0}, \tilde F_{k} = I-A\tilde X_{k,p},  \\
\tilde X_{k+1} & = X_{k,p} ((\tilde \alpha_k + \tilde \beta_k)I+\tilde \beta_k\tilde F_k),\textnormal{ for } k\in\{k_0+1,k_0+2,\ldots,k_0+m\},\\
\endaligned
$$
and finally
$$
\aligned
\tilde F_{k} &= I-A\tilde X_{k},\\
\tilde X_{k+1} &= \tilde X_{k} ((\tilde \alpha_k + \tilde \beta_k)I+\tilde \beta_k\tilde F_k),\qquad k>k_0+m.\\
\endaligned
$$
Here $\tilde \alpha_k$ and $\tilde \beta_k$ are determined using Algorithm \ref{optcoefit} for $\tilde F_k$ and $\tilde G_k=\tilde F_k^2$,
for every for $k\geqslant k_0$.
\end{dfn}

Simply put, $\seqn{\tilde X_k}$ is the perturbed iterative process obtained via the addition of a certain (perturbation) matrix to each $\tilde X_k$, for $k\in\{k_0+1,k_0+2,\ldots,k_0+m\}$, before the subsequent iteration is calculated. The theorem that follows demonstrates that in the case when $\seqn{X_k}$ does not converge to $A^{-1}$, there exist finitely many perturbation matrices whose norm can be arbitrarily small such that the iterative process modified by these perturbaations converges in a finite number of iterations.

\begin{thm}\label{Stabilnost}
Let $\seqn{X_k}$ the sequence defined by \eqref{algor0} where $\displaystyle\lim_{k\to \infty} F_k \neq 0$. If there exists a non-negative integer $l$ and a positive constant $m$ such that $\beta_k>m$ for all $k\geq l$, then for every $\epsilon>0$, there exists $k_0\in \N_0$ and the perturbation matrices $\Delta_1,\Delta_2,\ldots,\Delta_m$ for which $\|\Delta_i\|_F<\epsilon$, $i= 1,\ldots,m$ and  $\tilde F_{k_0+m+1}=0$.

\end{thm}
\begin{proof}
From the conditions of the theorem, and from the proof of Theorem \ref{PrvaT} we conclude that $\seqn{F_k}$ converges $F$, which is a projection, and  
$F\neq 0$.  
 We can assume without loss of generality that the first $m$ eigenvalues of $F$ are 1 and other $n-m$ are equal 0. This means that $F=PDP^T$ where $P$ is an orthogonal matrix and 
$$
D=\diag\left(\underbrace{1,1,\ldots,1}_m,\underbrace{0,\ldots,0}_{n-m}\right).
$$
Let $\epsilon >0$ be arbitrary, let $\delta = \min\left\{\frac{\epsilon}{ \|A^{-1}\|_F}, \epsilon\right\}$, and let $k_0$ be the non-negative integer for which 
\bb
\|F_{k_0+1}-F\|_F<\frac{\delta}2
\label{eq:Thm36-1}
\ee
Next, let
$$
D_0=\diag\left(1-\frac{\delta}{2},\underbrace{1,1,\ldots,1}_{m-1},\underbrace{0,\ldots,0}_{n-m}\right),
$$
and $\Delta_1 = A^{-1}(F_{k_0+1} - PD_0P^T)$. We have the equality $\tilde F_{k_0+1}=PD_0P^T$ and 
\bb
\| \tilde F_{k_0+1} - F\|_F =  \frac{\delta}{2}.
\label{eq:Thm36-2}
\ee
From \eqref{eq:Thm36-1} and \eqref{eq:Thm36-2} we get
$$
\| \tilde F_{k_0+1} - F_{k_0+1}\|_F  \leqslant \| \tilde F_{k_0+1} - F\|_F + \|F - F_{k_0+1}\|_F <\delta$$ and so $$
\|\Delta_1\|_F \leqslant \|A^{-1}\|_F\| \tilde F_{k_0+1} - F_{k_0+1}\|_F < \|A^{-1}\|_F\cdot \delta \leqslant\epsilon.
$$
For this choice, one gets
\[
\begin{aligned}
\tr(I-\tilde F_{k_0+1})
&=\frac{\delta}{2}+n-m,
&\qquad
\|I-\tilde F_{k_0+1}^2\|_F^2
&=\left(\delta-\frac{\delta^2}{4}\right)^2+n-m,\\[1ex]
\tr(I-\tilde F_{k_0+1}^2)
&=\delta-\frac{\delta^2}{4}+n-m,
&\qquad
\tr\!\left((I-\tilde F_{k_0+1})(I-\tilde F_{k_0+1}^2)\right)
&=\frac{\delta}{2}\left(\delta-\frac{\delta^2}{4}\right)+n-m,\\[1ex]
\|I-\tilde F_{k_0+1}\|_F^2
&=\frac{\delta^2}{4}+n-m.
\end{aligned}
\]
Using \eqref{alfabeta1} (the formulas \eqref{alfa2} and \eqref{beta2} are not applicable due to the pertubations), gives us the coefficients $\tilde \alpha_{k_0+1}$ and $\tilde \beta_{k_0+1}$:
$$
\begin{aligned}
 \tilde \alpha_{k_0+1}& = \frac{ \| I-\tilde F_{k_0+1}^2\|_F^2 \tr (I-\tilde F_{k_0+1}) - \tr(I-\tilde F_{k_0+1}^2)\tr ( (I-\tilde F_{k_0+1}^2)(I-\tilde F_{k_0+1})) }{ \| I- \tilde F_{k_0+1}\|_F^2 \| I- \tilde F_{k_0+1}^2\|_F^2 - \left( \tr ( (I-\tilde F_{k_0+1}^2)(I-\tilde F_{k_0+1})) \right)^2}= 1 - \frac{2}{\delta}, \\
 \tilde \beta_{k_0+1} & = \frac{ \| I-\tilde F_{k_0+1}\|_F^2 \tr (I-\tilde F_{k_0+1}^2) - \tr(I-\tilde F_{k_0+1})\tr ( (I-\tilde F_{k_0+1}^2)(I-\tilde F_{k_0+1})) }{ \| I- \tilde F_{k_0+1}\|_F^2 \| I- \tilde F_{k_0+1}^2\|_F^2 - \left( \tr ( (I-\tilde F_{k_0+1}^2)(I-\tilde F_{k_0+1})) \right)^2} = \frac{2}{\delta}
\end{aligned}.$$
The matrix $\bar F_{k_0+2}=I-A\tilde X_{k_0+2}$ has the following eigenvalues:
$$\begin{aligned} \bar \lambda_{1,k_0+1} & = 1- \left( \tilde \alpha_{k_0+1} + \tilde \beta_{k_0+1} \right) + \tilde \alpha_{k_0+1} \bar \lambda_{1,k_0} + \tilde \beta_{k_0+1} \bar \lambda_{1,k_0}^2 = \\
&  = 0 + \left( 1- \frac{2}{\delta} \right) \left(1- \frac{\delta}{2} \right) + \frac{2}{\delta} \left( 1- \frac{\delta}{2} \right)^2 = \\
& =  \left( 1 - \frac{\delta}{2} \right) \left( 1- \frac{2}{\delta} + \frac{2}{\delta} -1\right) =  0, \\
\bar \lambda_{j,k_0+1} & = 1, \qquad j\in\{2,\dots,m\}, \\
\bar \lambda_{j,k_0+1} & = 0, \qquad j \in\{m+1, \dots ,n\}.\end{aligned}$$
If
$$
\Delta_i=A^{-1}P\,\diag\left(\underbrace{0,0,\ldots,0}_{i-1},\delta,\underbrace{0,0,\ldots,0}_{n-i}\right)P^T, \quad
i\in\{1,2,\ldots,m\}
$$
it follows that $\|\Delta_i\|_F<\|A^{-1}\|_F\delta\leqslant \epsilon$, and  
the eigenvalues of $\bar F_{k_0+i+1}=I-A\tilde X_{k_0+i+1}$ are
$$\begin{aligned} \bar \lambda_{k_0+1,i} & = 0, \qquad j\in \{1,2,\ldots,i\}, \\
\bar \lambda_{j,k_0+i} & = 1, \qquad j\in\{i+1,\dots,m\}, \\
\bar \lambda_{j,k_0+i} & = 0, \qquad i \in\{m+1, \dots ,n\}.\end{aligned}$$
The conclusion now follows from $\tilde F_{k_0+m+1} = \bar F_{k_0+m+1} = 0$.
\end{proof}

Theorem \ref{Stabilnost} suggests that though the \textbf{SSHP2} iterative method $\seqn{X_k}$ isn't convergent, the perturbation(s) (which are arbitrarily small) can alter the behaviour of the method into convergence. These perturbation(s) could originate from the round-off errors, which naturally occur in floating point arithmetic. While Theorem \ref{Stabilnost} has predetermined perturbations, a similar result might be proved for a much wider, or possibly even arbitrary choice of perturbation matrices.

\section{Numerical results and performance testing}
\label{sect:testing}

This section contains the results of multiple numerical experiments in which the method \textup{\textbf{SSHP2}} was analysed together with a number of well known iterative methods used to find the matrix inverse, including \textup{\textbf{OPM}}. The implementation of the method and the tests were done in Python. As a measure of precaution against catastrophic cancellation a relative tolerance was used in the computation of the coefficients $\alpha_k$ and $\beta_k$. To be more precise, instead of the condition $|D_k|\geqslant \delta$, the condition $|D_k| \geqslant \delta\cdot \left|C_{00}^{(k)}C_{11}^{(k)}\right|$ is used, where $\delta >0$ is the relative tolerance. In this case we have used $\delta= 10^{-10}$ as the relative tolerance.
\bigskip

\begin{exm}\label{primer1}
In this experiment the performance of the method \textup{\textbf{SSHP2}} was examined by applying it on \href{https://nhigham.com/2021/07/06/what-is-the-kac-murdock-szego-matrix/}{Kac-Murdock-Szeg\"o} matrices of parameter close to $1$, which are a familiar class of ill-conditioned matrices, and the results compared to the results when the standard Schultz method and the \textup{\textbf{OPM}} method were used. 
    Let $\textup{KMS}_{n,\rho}$ denote the Kac-Murdock-Szeg\"o matrix whose parameter is $\rho$ of order $n$. Furhtermore, the methods  \textup{\textbf{SSHP2}} and \textup{\textbf{OPM}} were used with two choices of initial approximation, $X_0=\frac{2}{\| A \|_F^2} A^T$ and $X_0=A^T$. The choice for the parameter $\rho$ was $\rho=0.99$, and the order was $n= 200,300,400,500,600$, while Schultz's method was tested with the usual initial approximation $X_0=\frac{2}{\|A \|_F^2} A^T$. The number of iterations and matrix multiplications until the precision of $\varepsilon=10^{-10}$ was reached for each test is presented in Table \ref{tbl:1}.

 \begin{table}[H]
    $$
    \begin{array}{|| c || c || c || c || c || c || c ||}
    \hline
        \textup{Matrix} & ~ & \textbf{HP2} & \begin{aligned}  &\textbf{\textup{OPM}}\\ X_0 &=\frac{2}{\|A\|_F^2}A^T \end{aligned} & \begin{aligned}  &\textbf{\textup{OPM}}\\ X_0 &=A^T \end{aligned} & \begin{aligned}  &\textbf{\textup{SSHP2}}\\ X_0 &=\frac{2}{\| A \|_F^2}A^T \end{aligned} & \begin{aligned}  &\textbf{\textup{SSHP2}}\\ X_0 &=A^T \end{aligned} \\ \hline
        \textup{KMS}_{200,0.99} & \textup{Iterations} & 33 & 21 & 31 & \textbf{20} & \textbf{20} \\ \hline
        ~ & \textup{MMS} & 66 & 63 & 93 & \textbf{60} & \textbf{60} \\ \hline
        \textup{KMS}_{300,0.99} & \textup{Iterations} & 34 & 22 & 51 & \textbf{20} & \textbf{20} \\ \hline
        ~ & \textup{MMS} & 68 & 66 & 153 & \textbf{60} & \textbf{60} \\ \hline
        \textup{KMS}_{400,0.99} & \textup{Iterations} & 34 & 22 & 50 & \textbf{21} & \textbf{21} \\ \hline
        ~ & \textup{MMS} & 68 & 66 & 150 & \textbf{63} & \textbf{63} \\ \hline
        \textup{KMS}_{500,0.99} & \textup{Iterations} & 35 & 22 & 58 & \textbf{21} & \textbf{21} \\ \hline
        ~ & \textup{MMS} & 70 & 66 & 174 & \textbf{63} & \textbf{63} \\ \hline
        \textup{KMS}_{600,0.99} & \textup{Iterations} & 35 & 22 & 31 & \textbf{21} & \textbf{21} \\ \hline
        ~ & \textup{MMS} & 70 & 66 & 93 & \textbf{63} & \textbf{63} \\ \hline
    \end{array}
    $$
    \caption{Number of iterations and matrix multiplications required to reach the preferred error tolerance of $\varepsilon=10^{-10}$ for different values of $n$.}
\label{tbl:1}
\end{table}   

We should remark that the choice of  $X_0$ did not impact the convergence of \textbf{\textup{SSHP2}} and \textup{\textbf{OPM}}. Other known iterative methods are often not flexible in the choice of the initial approximation, and it shows that \textbf{\textup{SSHP2}}, similarly to \textup{\textbf{OPM}}, is converges for much broader choice of the initial matrix $X_0\in \R^{n\times n}$. More significantly, \textup{\textbf{SSHP2}} was the most efficient method for each choice of $X_0$.

Somewhat surprisingly, unlike \textup{\textbf{OPM}}, the different choice of initial approximation had little effect on the number of iterations (equivalently matrix multiplications) that were needed to reach a precision of $\varepsilon=10^{-10}$ when \textup{\textbf{SSHP2}} was tested. Table \ref{tbl:2} shows the values of 
the residuals for each iteration of \textup{\textbf{SSHP2}} applied to $\textup{KMS}_{600,0.99}$ for both choice of starting approximations. It shows that even though the initial residuals were very different the residuals that followed were very close.
\begin{table}[H]
\centering
\begin{minipage}{0.48\textwidth}
\centering
\[
\begin{array}{||c||c||c||}
\hline
~ & X_0=\frac{2}{\| A \|_F^2}A^T & X_0=A^T\\
\hline
\textnormal{Iteration} & \|F_k\|_F & \|F_k\|_F\\
\hline
0  & 2.444\cdot10^{1} & 3.350\cdot10^{4}\\ \hline
1  & 2.441\cdot10^{1} & 2.441\cdot10^{1}\\ \hline
2  & 2.438\cdot10^{1} & 2.438\cdot10^{1}\\ \hline
3  & 2.433\cdot10^{1} & 2.433\cdot10^{1}\\ \hline
4  & 2.427\cdot10^{1} & 2.427\cdot10^{1}\\ \hline
5  & 2.419\cdot10^{1} & 2.419\cdot10^{1}\\ \hline
6  & 2.408\cdot10^{1} & 2.408\cdot10^{1}\\ \hline
7  & 2.393\cdot10^{1} & 2.393\cdot10^{1}\\ \hline
8  & 2.372\cdot10^{1} & 2.372\cdot10^{1}\\ \hline
9  & 2.344\cdot10^{1} & 2.344\cdot10^{1}\\ \hline
10 & 2.304\cdot10^{1} & 2.304\cdot10^{1}\\
\hline
\end{array}
\]
\end{minipage}
\hfill
\begin{minipage}{0.48\textwidth}
\centering
\[
\begin{array}{||c||c||c||}
\hline
~ & X_0=\frac{2}{\| A \|_F^2}A^T & X_0=A^T\\
\hline
\textnormal{Iteration} & \|F_k\|_F & \|F_k\|_F\\
\hline
11 & 2.249\cdot10^{1} & 2.249\cdot10^{1}\\ \hline
12 & 2.167\cdot10^{1} & 2.167\cdot10^{1}\\ \hline
13 & 2.044\cdot10^{1} & 2.044\cdot10^{1}\\ \hline
14 & 1.844\cdot10^{1} & 1.844\cdot10^{1}\\ \hline
15 & 1.508\cdot10^{1} & 1.508\cdot10^{1}\\ \hline
16 & 9.626\cdot10^{0} & 9.626\cdot10^{0}\\ \hline
17 & 3.498\cdot10^{0} & 3.498\cdot10^{0}\\ \hline
18 & 4.272\cdot10^{-1} & 4.272\cdot10^{-1}\\ \hline
19 & 7.033\cdot10^{-3} & 7.033\cdot10^{-3}\\ \hline
20 & 2.322\cdot10^{-6} & 2.322\cdot10^{-6}\\ \hline
21 & 8.310\cdot10^{-12} & 8.240\cdot10^{-12}\\
\hline
\end{array}
\]
\end{minipage}

\caption{Values of the residuals after each iteration of \textup{\textbf{SSHP2}} applied to $\textup{KMS}_{600,0.99}$ for different initial approximations.}
\label{tbl:2}
\end{table}

\end{exm}

\begin{exm}\label{primer2}
 In this experiment the performance of \textup{\textbf{SSHP2}}  is examined together with the following iterative methods, \textbf{\textup{HP2}}, \textbf{\textup{HP3}},  \textbf{\textup{IHP9}} and \textup{\textbf{OPM}} (see \cite{MP2} regarding \textbf{\textup{IHP9}}) on 10 random matrices whose dimensions are $n\times n$ where $n\in \lbrace 800, 1000, 1200, 1400, 1600, 1800, 2000\rbrace$. Comparisons were made in two categories, the average number of \textup{MMS} (matrix multiplications), and the average time of execution.  As in the previous experiment, the precision was $\varepsilon=10^{-10}$. For consistency, an invertibility check was performed before each trial. Each entry of the matrices used in the experiment is generated by the random sample with  uniform distribution on $[-1,1]$. The initial approximation is chosen to be $X_0=\frac{2}{\| A \|_F^2} A^T$. Table \ref{tbl:3} and Table \ref{tbl:4} show the obtained results. Table \ref{tbl:3} shows the average number of matrix multiplications until the required precision was reached, while Table \ref{tbl:4} contains the average time of execution.

 \begin{table}[H]
   \centering
    \begin{tabular}{|| c || c || c || c || c || c ||}
    \hline
        \textbf{MatrixSize} & \textbf{HP2} & \textbf{HP3} & \textbf{IHP9}  & \textbf{OPM} & \textbf{SSHP2} \\ \hline
         800 & 69.4 & 69 & 66 & 61.5 & \textbf{57.6} \\ \hline
        1000 & 74 & 73.8 & 70 & 63.9 & \textbf{61.8} \\ \hline
        1200 & 74.6 & 74.1 & 70 & \textbf{64.5} & \textbf{64.5} \\ \hline
        1400 & 71.4 & 70.8 & 67.5 & 60.9 & \textbf{57.6} \\ \hline
        1600 & 74.6 & 74.1 & 71 & 67.5 & \textbf{60.6} \\ \hline
        1800 & 75.6 & 74.7 & 71.5 & 66.6 & \textbf{63.3} \\ \hline
        2000 & 77.4 & 77.1 & 73 & 70.5 & \textbf{64.2} \\ \hline
    \end{tabular}
   \caption{Average total number of MMS required  to approximate the inverse to the required precision $\varepsilon=10^{-10}$. 
   }
\label{tbl:3} 
\end{table}
\begin{table}[H]
    \centering
    \begin{tabular}{|| c || c || c || c || c || c ||}
    \hline
        \textbf{MatrixSize} & \textbf{HP2} & \textbf{HP3} & \textbf{IHP9}  & \textbf{OPM} & \textbf{SSHP2} \\ \hline
      800 & \textbf{1.192} & 1.308 & 1.337 & 1.206 & 1.228 \\ \hline
        1000 & 1.638 & \textbf{1.587} & 1.932 & 1.659 & 1.642 \\ \hline
        1200 & 3.018 & \textbf{2.342} & 2.896 & 3.129 & 3.772 \\ \hline
        1400 & 4.492 & 4.371 & 4.290 & \textbf{3.523} & 3.568 \\ \hline
        1600 & 5.299 & 9.232 & 8.715 & 5.340 & \textbf{4.834} \\ \hline
        1800 & 9.018 & 8.303 & 8.070 & 9.151 & \textbf{7.669} \\ \hline
        2000 & 9.564 & 17.493 & 13.563 & 9.716 &\textbf{9.200} \\ \hline
    \end{tabular}
    \caption{Average execution time (in seconds) required to approximate the inverse to the required accuracy $\varepsilon=10^{-10}$. 
    }
\label{tbl:4} 
\end{table}
From Table \ref{tbl:3} we see that in terms of the total number of matrix multiplications \textup{\textbf{SSHP2}} is the most efficient method, while Table \ref{tbl:4} demonstrates that for random matrices of lower sizes \textup{\textbf{SSHP2}} is slightly slower than \textup{\textbf{HP2}} or \textup{\textbf{HP3}} but becomes faster as the dimensions increase. This is most like due to the implementation in Python, which creates a time overhead in the calculation in the coefficients $\alpha_k$ and $\beta_k$. Although these calculations (using Algorithm \ref{optcoefit}) are performed in $\O(n^2)$ (which is 
faster than matrix multiplications where the complexity is $\Omega(n^2)$), manual implementation in Python can produce the additional time overhead.
\end{exm}

\begin{exm}\label{primer3}
    The goal of this experiment was to test the performance of \textbf{\textup{SSHP2}} on some ill-conditioned matrices from \textup{\texttt{MatrixMarket}} \cite{MatrixMarket}, and compare it to the performance of \textbf{\textup{HP2}}, \textbf{\textup{HP3}}, \textbf{\textup{IHP9}} and \textup{\textbf{OPM}}. The comparisons were again made using \textup{MMS} and execution time until the precision $\varepsilon=10^{-10}$ was reached. Results are shown in Table \ref{tbl:5}. 

 \begin{table}[H]
$$ 
\begin{array}{||c||c||c||c||c||c ||}\hline
 \textbf{\textup{Matrix}} & \textbf{\textup{HP2}} & \textbf{\textup{HP3}} & \textbf{\textup{IHP9}} & \textbf{\textup{OPM}} & \textbf{\textup{SSHP2}} \\ \hline
\href{https://math.nist.gov/MatrixMarket/data/Harwell-Boeing/smtape/bp___200.html}{\texttt{\textup{bp\_\_\_200}}} & & & & & \\ \hline
\textup{MMS}&	104 &	102 &	\textbf{95} & 222 & 129\\ \hline
\textup{TIME}&	1.749 &	\textbf{1.568} &	1.811	& 3.568	& 2.175
 \\ \hline
\href{https://math.nist.gov/MatrixMarket/data/Harwell-Boeing/smtape/bp___400.html}{\texttt{\textup{bp\_\_\_400}}} & & & & & \\ \hline
\textup{MMS} &	102	& 99 & \textbf{90} & 96 & 117\\ \hline
\textup{TIME} &	1.411 &	\textbf{1.348}	& 1.617 & 1.566	& 2.020 
\\ \hline
\href{https://math.nist.gov/MatrixMarket/data/SPARSKIT/fidap/fidap001.html}{\texttt{\textup{fidap01}}} & & & & & \\ \hline
\textup{MMS}	& 78	& 78 &	75	& 72 & \textbf{69} \\ \hline
\textup{TIME}	& 0.0195 &	0.0187 & 0.0198 & 0.015135
& \textbf{0.015130}
\\ \hline
\href{https://math.nist.gov/MatrixMarket/data/NEP/mvmmcd/cdde1.html}{\texttt{\textup{cdde1}}} & & & & & \\ \hline
\textup{MMS} &	70 & 69 & 65 & 60 & \textbf{57}\\ \hline
\textup{TIME} &	1.373	& \textbf{1.303} & 	1.664	& 1.429 &	1.391  \\ \hline
\href{https://math.nist.gov/MatrixMaket/data/NEP/mvmmcd/cdde2.html}{\texttt{\textup{cdde2}}} & & & & & \\ \hline
\textup{MMS} & 48 &	48 & 50 & 39 & \textbf{36}\\ \hline
\textup{TIME} &	0.958 &	0.90325 
&	1.294 &	0.925	& \textbf{0.90320} 
\\ \hline
\href{https://math.nist.gov/MatrixMarket/data/NEP/mvmmcd/cdde3.html}{\texttt{\textup{cdde3}}} & & & & & \\ \hline
\textup{MMS} & 72 & 72 & 70 & 63 & \textbf{60}\\ \hline
\textup{TIME} & 1.406 &	\textbf{1.395}	& 1.814 & 1.650	& 1.493  \\ \hline
\href{https://math.nist.gov/MatrixMarket/data/NEP/mvmmcd/cdde4.html}{\texttt{\textup{cdde4}}} & & & & & \\ \hline
\textup{MMS} & 50 &	51 & 50 & 39 & \textbf{36} \\ \hline
\textup{TIME} & 1.006 & 0.974 & 1.317  & 0.942 & \textbf{0.915}  \\ \hline
\href{https://math.nist.gov/MatrixMarket/data/NEP/mvmmcd/cdde6.html}{\texttt{\textup{cdde6}}} & & & & & \\ \hline
\textup{MMS} & 56 & 57 & 55 & 45 & \textbf{42} \\ \hline
\textup{TIME} & 1.126 & 1.104 & 1.520 & 1.098& \textbf{1.029}   \\ \hline

\end{array}
$$
\caption{Total number of matrix multiplications (MMS) and time of execution (TIME, in seconds) for the tested methods and different ill-conditioned matrices from \texttt{\textup{Matrix Market}} \cite{MatrixMarket}.
}
\label{tbl:5}
\end{table}
According to Table \ref{tbl:5}, \textup{\textbf{SSHP2}} is the most efficient method in terms of the number of matrix multiplications in most cases, and that it is often the fastest method (in terms of total running time). As previously mentioned, one possible reason for the time difference is the manual implementation of the Algorithm \ref{optcoefit} for computing coefficients $\alpha_k$ and $\beta_k$ in Python, as discussed in the end of the previous example.

Method \textup{\textbf{SSHP2}} was the best in terms of MMS in 6 out of 8 tests, and the fastest in 4 out of 8 tests. 
 \end{exm}

\section{Conclusion and further research}
\label{sect:conclusion}

The constructed method \textbf{SSHP2} presents a new approach in computing of matrix inverses as the coefficients of the iterative method are dynamically computed in each iteration, by minimizing the Frobenius norm of the residual. As the objective function is quadratic in this case, we could derive explicit expressions to calculate the coefficients. 

\medskip

According to the numerical examples shown in the previous section, \textbf{SSHP2} method outperforms the most efficient existing iterative methods in the majority of test cases. More importantly, the examples suggest that it converges for the much wider choice of the initial approximations $X_0$ then classical iterative methods (they all require $r(I-AX_0)<1)$), and possibly posses \textbf{almost global convergence} (under the assumption that $F_0=I-AX_0$ is symmetric). The proof of the latter fact is remained as the open problem. Some other possible directions for further research are:
\begin{itemize}
    \item Is it possible to generalize this approach to the computation of the Moore--Penrose and other generalized inverses, for non-regular matrices $A$?
    \item Is it possible to create general scheme for construction of methods with any number of coefficients and how would those methods work in practice, while having in mind that expressions are complicated even in the simplest case?
   \item Which method from that class is computationally optimal? The last question is naturally imposed, since with adding coefficients, we are increasing speed of convergence while the complexity and number of calculations in each iteration is increased as well.
\end{itemize} 


\section*{Acknowledgement}

The authors were financially supported by the Ministry Of Science, Technological Development and Innovation of the Republic of Serbia (Contracts No. 451-03-34/2026-03/200109, 451-03-33/2026-03/200104, 451-03-34/2026-03/200124).

\medskip

The fourth author is supported by the bilateral grant between Serbia and Germany no. 337-00-253/2023-05/14, by the Ministry Of Science, Technological Development, and Innovation of the Republic of Serbia, and the DAAD foundation (Germany). 

\section*{Competing interests} 

The authors declare that they have no competing interests.


\begin{thebibliography}{99}\setlength{\parskip}{2pt}\setlength{\itemsep}{0pt plus 0.4ex}
{\footnotesize









\bibitem{Aceto2025} L. Aceto, L. Gemignani, \textit{A scaling-and-squaring method for computing the inverses of matrix $\varphi$-functions}, Numerical Algrotihms \textbf{100 (4)} (2025), 1581-1598. \href{https://doi.org/10.1007/s11075-025-02049-1}{\texttt{https://doi.org/10.1007/s11075-025-02049-1}}

\bibitem{Seven} D. Ahmed, M. Hama, K. H. Faraj Jwamer, S. Shateyi, \textit{A Seventh-Order Scheme for Computing the Generalized Drazin Inverse}, Mathematics {\bf 7(7)} (2019),  622;
\href{https://doi.org/10.3390/math7070622}{\texttt{https://doi.org/10.3390/math7070622}}

\bibitem{Art} S. Artidiello, A. Cordero, J. R. Torregrosa, M. P. Vassileva, \textit{Generalized Inverses Estimations by Means of Iterative Methods with Memory}, Mathematics {\bf 8(1)} (2020),  2;  \href{https://doi.org/10.3390/math8010002}{\texttt{https://doi.org/10.3390/math8010002}}

\bibitem{MIMO} M. Ataeeshojai, R. C. Elliott, W. A. Krzymień, C. Tellambura, I. Maljević, \textit{Iterative Matrix Inversion Methods for Precoding in Cell-Free Massive MIMO Systems},  IEEE Transactions on Vehicular Technology  {\bf 71 (11)} (2022),
\href{https://doi.org/10.1109/TVT.2022.3194870}{\texttt{https://doi.org/10.1109/TVT.2022.3194870}}

\bibitem{BB} F. S.V. Bazán,  E. Boos, \textit{Schultz matrix iteration based method for stable solution of discrete ill-posed problems}, Linear Algebra and its Applications {\bf 554} (2018), 120--145. \href{https://doi.org/10.1016/j.laa.2018.05.022}{\texttt{https://doi.org/10.1016/j.laa.2018.05.022}}

\bibitem{Cordero2021} A. Cordero, P. Soto-Quiros, J. R. Torregrossa, \textit{A general class of arbitrary order iterative methods for computing generalized inverses}, Applied Mathematics and Computation \textbf{409} (2021), 126381. \href{https://doi.org/10.1016/j.amc.2021.126381}{\texttt{https://doi.org/10.1016/j.amc.2021.126381}}

\bibitem{Cordero2024} A. Cordero, E. Segura, J. R. Torregrossa, M. P. Vassileva, \textit{Inverse matrix estimations by iterative methods with weight functions and their stability analysis}, Applied Mathematics Letters \textbf{155} (2024), 109122. \href{10.1016/j.aml.2024.109122}{\texttt{10.1016/j.aml.2024.109122}}

\bibitem{ChenWan} H. Chen, Y. Wang, {\em A Family of higher-order convergent iterative methods for computing
the Moore–-Penrose inverse}, Applied Mathematics and Computation {\bf 218} (2011), 4012-–4016, \href{https://doi.org/10.1016/j.amc.2011.05.066}{\texttt{https://doi.org/10.1016/j.amc.2011.05.066}}

\bibitem{SMSMP} L. Chen, E.V. Krishnamurthy, I. Macleod, {\em Generalized matrix inversion and rank computation by successive matrix powering}, Parallel Computing {\bf 20} (1994), 297--311, \href{https://doi.org/10.1016/S0167-8191(06)80014-1}{\texttt{https://doi.org/10.1016/S0167-8191(06)80014-1}}

\bibitem{GIBS} E. K. Dehdezi, S. Karimi, \textit{GIBS: a general and efficient iterative method for computing the approximate inverse and Moore–Penrose inverse of sparse matrices based on the Schultz iterative method with applications}, Linear and Multilinear Algebra {\bf 71} (2023), 1905-1921. 
\href{https://doi.org/10.1080/03081087.2022.2088673}{\texttt{https://doi.org/10.1080/03081087.2022.2088673}}

\bibitem{Horn} R. E. Horn, C. R. Johnson, \textit{Matrix Analysis - 2nd ed.}, Cambridge University Press, 2012.

\bibitem{Janna2025} C. Janna, A. Franceschini, \textit{Nesting Approximate Inverses for Improved Preconditioning and Algebraic Multigrid Smoothing}, SIAM Journal on Matrix Analysis and Applications \textbf{46 (1)} (2025), 393-415. \href{https://doi.org/10.1137/24M1679847}{\texttt{https://doi.org/10.1137/24M1679847}}

\bibitem{Kaur2020} M. Kaur, M. Kansal, S. Kumar, \textit{An efficient hyperpower iterative method for computing weighted MoorePenrose inverse}, AIMS Mathematics \textbf{5 (3)} (2020), 1680-1692. \href{10.3934/math.2020113}{\texttt{10.3934/math.2020113}}

\bibitem{effgen} M. Kansal, S. Kumar, M. Kaur, \textit{An efficient matrix iteration family for finding the generalized outer inverse}, Applied Mathematics and Computation {\bf 430} (2022), 127292, \href{https://doi.org/10.1016/j.amc.2022.127292}{\texttt{https://doi.org/10.1016/j.amc.2022.127292}} 

\bibitem{OPM} M. Kostadinov, M. Krsti\' c, K. Rajkovi\' c, M.D. Petkovi\' c, \emph{Adaptive coefficients iterative method for computing matrix inverse}, Linear Algebra and its Applications, 731 (2026) 277-305.

\bibitem{XiaLi} C. Li, L. Yong, \textit{Modified BAS iteration method for absolute value equation}, AIMS Mathematics {\bf 7(1)} (2022), 606--616.
\href{https://doi.org/10.3934/math.2022038}{\texttt{https://doi.org/10.3934/math.2022038}}

\bibitem{Liu} X. Liu, H. Jin, Y. Yu, {\em Higher-order convergent iterative method for computing the generalized inverse and its application to Toeplitzz matrices}, Linear Algebra and its Applications, 439 (2013), 1635-1650.

\bibitem{Martinez} A. Marco, J. J. Martinez, \textit{Accurate computation of the Moore-Penrose inverse of strictly totally positive matrices}, Journal of Computational and Applied Mathematics {\bf 350} (2019), 299-308.  
\href{https://doi.org/10.1016/j.cam.2018.10.009}{\texttt{https://doi.org/10.1016/j.cam.2018.10.009}}

\bibitem{MatrixMarket} Matrix Market, National Institute of Standards and Technology, Gaithersburg, MD. Available online from \href{https://math.nist.gov/MatrixMarket/}{\texttt{https://math.nist.gov/MatrixMarket/}}

\bibitem{MP1} M.D. Petkovi\' c,  \textit{Generalized Schultz iterative methods for the computation of outer inverses}, Computers \& Mathematics with Applications {\bf 67 (10)} (2014), 1837–1847. \href{https://doi.org/10.1016/j.camwa.2014.03.019}{\texttt{https://doi.org/10.1016/j.camwa.2014.03.019}}

\bibitem{MP2} M.D. Petkovi\' c, M.S. Petkovi\' c, \textit{Hyper–power methods for the computation of outer inverses}, Journal of Computational and Applied Mathematics {\bf 278} (2015) 110–118. \href{https://doi.org/10.1016/j.cam.2014.09.024}{\texttt{https://doi.org/10.1016/j.cam.2014.09.024}}

\bibitem{Sch} Schulz G., \textit{Iterative Berechnung Der Reziproken Matrix}, Zeitschrift für Angewandte Mathematik und Mechanik
{\bf 13} (1933), 57-59.   \href{https://doi.org/10.1002/zamm.19330130111}{\texttt{https://doi.org/10.1002/zamm.19330130111}}

\bibitem{SMSDrazin} Y. Wei, {\it Successive matrix squaring algorithm for computing the Drazin inverse}, Applied Mathematics and Computation {\bf 108} (2000), 67--75, \href{https://doi.org/10.1016/S0096-3003(98)10118-2}{\texttt{https://doi.org/10.1016/S0096-3003(98)10118-2}}

\bibitem{Weiguo} L. Weiguo, L. Juan, Q. TianTian, {\it A family of iterative methods for computing Moore–-Penrose inverse of a matrix}, Linear Algebra and its Applications {\bf 438} (2013), 47–56, \href{https://doi.org/10.1016/j.laa.2012.08.004}{\texttt{https://doi.org/10.1016/j.laa.2012.08.004}}

\bibitem{PKR} M. D. Petkovi\'c, M. A. Krsti\'c, K. P. Rajkovi\'c, {\it Rapid generalized Schultz iterative methods for the computation of outer inverses}, Journal of Computational and Applied Mathematics {\bf 344} (2018), 572--584, \href{https://doi.org/10.1016/j.cam.2018.05.048}{\texttt{https://doi.org/10.1016/j.cam.2018.05.048}}.

\bibitem{Yang2023} K. Yang, X. Tan, Q. Qian, Z. Zhang, X. You, C. Zhang, \textit{An Improved Newton-Schulz Iterative Algorithm for Massive MIMO Detection}, 2023 International Conference on Wireless Communications and Signal Processing (WCSP), \href{10.1109/WCSP58612.2023.10404466}{\texttt{10.1109/WCSP58612.2023.10404466}}

\bibitem{ZhangKangTan} J. Zhang, H. Kang and F. Tan, \textit{Two-parameters numerical methods of the non-symmetric algebraic Riccati equation}, Journal of Computational and Applied Mathematics {\bf 378} (2020), 112933;
\href{https://doi.org/10.1016/j.cam.2020.112}{\texttt{https://doi.org/10.1016/j.cam.2020.112}}

\bibitem{Zhou2025} X. Zhou, S. Fan, M. Jaggi, \textit{HyperINF: Unleashing the HyperPower of the Schulz's Method for Data Influence Estimation}, Arxiv preprint. \href{
https://doi.org/10.48550/arXiv.2410.05090}{\texttt{
https://doi.org/10.48550/arXiv.2410.05090}}



}

\end{thebibliography}
\end{document}